\documentclass[11pt,reqno]{article}
\usepackage[margin=1in]{geometry}
\usepackage{dsfont, amssymb,amsmath,amscd,latexsym, amsthm, amsxtra,amsfonts}
\usepackage{lineno}
\usepackage[active]{srcltx}

\usepackage{tikz}
\usepackage{natbib}
\usepackage{bbm}
\usepackage{enumerate}
\usepackage{mathrsfs}
\usepackage{graphicx}
\usepackage{comment}
\usepackage{mathtools}
\usepackage{cases}
\usepackage{tikz}
\usetikzlibrary{calc,arrows}
\usepackage{verbatim}
\usepackage{graphicx}
\usepackage{subfigure}
\usepackage{color}

\usepackage{epstopdf}
\usepackage{algorithm,algorithmic}

\newtheorem{theorem}{Theorem}[section]
\newtheorem{corollary}[theorem]{Corollary}
\newtheorem{proposition}[theorem]{Proposition}
\newtheorem{lemma}[theorem]{Lemma}

\theoremstyle{definition}
\newtheorem{definition}[theorem]{Definition}

\theoremstyle{definition}

\theoremstyle{definition}
\newtheorem{remark}{Remark}
\theoremstyle{definition}

\renewcommand{\epsilon}{\varepsilon}

\usepackage[pdfstartview=FitH, bookmarksnumbered=true,bookmarksopen=true, colorlinks=true, pdfborder={0 0 1}, citecolor=blue, linkcolor=blue,urlcolor=blue]{hyperref}
\usepackage{graphics}
\graphicspath{{figures/}}
\title{Reinforcement learning for irreversible reinsurance problems:\\ the randomized singular control approach}

\author{Zongxia Liang\thanks{Department of Mathematical Sciences, and China Center for Insurance and Risk Management of Tsinghua SEM,Tsinghua University, Beijing 100084, People’s Republic of China
  (\url{liangzongxia@tsinghua.edu.cn}).}
  \and  Xiaodong Luo\thanks{Department of Mathematical Sciences, Tsinghua University, Beijing 100084, People’s Republic of China
  (\url{luoxd21@mails.tsinghua.edu.cn}).}
  \and Xiang Yu\thanks{Department of Applied Mathematics, The Hong Kong Polytechnic University, Kowloon, Hong Kong (\url{xiang.yu@polyu.edu.hk})}
}
\numberwithin{equation}{section}

\usepackage{amsopn}

\begin{document}
\date{ }

\maketitle

\vspace{-0.2in}
\begin{abstract}

This paper studies the continuous-time reinforcement learning for stochastic singular control with the application to an infinite-horizon irreversible reinsurance problem. The singular control is equivalently characterized as a pair of regions of time and the augmented states, called the singular control law. To encourage the exploration in the learning procedure, we propose a randomization method by considering an auxiliary singular control and entropy regularization. The exploratory singular control problem is formulated as a two-stage optimal control problem, in which the time-inconsistency issue arises in the outer problem. Existence of equilibrium singular control law for the time-inconsistent outer problem is rigorously established. Taking advantage of the solution structure, we utilize a proper parameterization and neural networks to devise the actor-critic reinforcement learning algorithm. In the numerical experiment, we show the superior convergence of parameter iterations based on the randomized equilibrium policy and illustrate how the exploration may advance the learning performance.

\end{abstract}

\section{Introduction}

Reinforcement Learning (RL), a paradigm of machine learning that allows an agent to learn the optimal decision by interacting with an unknown environment, has received an upsurge of developments in both theories and applications. Traditional RL has achieved remarkable success by modeling the trial-and-error decision-making in discrete time framework. However, this discrete-time assumption is often a simplification of reality. Many real-world systems, particularly in finance, physics, and engineering, evolve continuously through time. Recently, continuous-time reinforcement learning algorithms have attracted more attention and spurred a lot of new and fruitful theoretical advances. A series of pioneer studies such as  \cite{wang2020continuous}, \cite{wang2020reinforcement}, \cite{jia2022policy}, \cite{jia2022policy1}, \cite{jia2023q} lay the foundation for continuous-time exploratory RL algorithms under entropy regularization from policy evaluation, policy gradient to q-learning algorithms. Later, more subsequent extensions of continuous-time RL algorithms under entropy regularization in various financial applications have been considered, see for instance,  \cite{dai2023learning}, \cite{DaiDJZ}, \cite{Jia24}, \cite{BoHuangYu}, \cite{WeiYu}, \cite{wyy2024}, \cite{wgl},
and references therein. 

Notably, the aforementioned studies only focus on RL approach for regular control policies, where the entropy regularization can be defined in a straightforward manner. Recently, the optimal stopping problem has witnessed some progress in the domain of RL approach. For instance, \cite{dong2024randomized} proposed an entropy regularization for the optimal stopping problem by considering the regular control of intensity to devise the RL algorithm. \cite{dai2024learning} considered the penalty approximation that leads to a regular control problem, where the entropy regularization is imposed on the regular control policy. \cite{HLYZ2025} studied the reinforcement learning for optimal switching under entropy regularization, for which the randomized policy is to control the generator matrix of an associated continuous-time Markov chain.
As these studies randomize the stopping time as the regular control, some previous results such as the martingale characterization in \cite{jia2022policy} and \cite{jia2022policy1} can be directly adopted. \cite{dianetti2024exploratory} recently considered the residual entropy on the probability that the agent stops before time $t$ and formulated the randomization of the stopping time as a singular control problem, in which they obtained a policy improvement result in updating the boundary. Later, \cite{dianetti2025entropy} extended this randomization method in solving the mean-field game of optimal stopping.

The existing study of continuous-time RL algorithms for singular control problems is rare. Noting the connection between some types of singular control and optimal stopping problems that has been discussed in some early studies, see \cite{karatzas1984connections}, \cite{karatzas1985connections} and \cite{karoui1991new},  recently \cite{dai2024learning} transformed the optimal singular control problem of optimal investment with transaction cost into a Dynkin game of optimal stopping to learn the optimal strategy using the techniques developed in \cite{dong2024randomized}. \cite{liang2025reinforcement} proposed a reinforcement learning framework for a class of singular control problems without entropy regularization by working with singular control laws and the established policy improvement result for region iterations. How to formulate a proper randomization of the singular control to facilitate the learning exploration is still an interesting open problem.

Building upon the idea of learning the region of singular control laws, this work aims to further consider the exploration and formalize the randomization of singular control laws by introducing another auxiliary singular control law. However, we note that the entropy regularized singular control problem faces the issue of time-inconsistency. The optimal solution of the original problem without regularization turns out to be an equilibrium solution of the entropy regularized problem. In our proposed formulation of entropy regularization, we need to learn the time-consistent equilibrium as a pair of singular control laws. In the specific setup for irreversible reinsurance problem, we are able to establish the existence of equilibrium. Further, we can show the convergence of the equilibrium value function towards the optimal value function of the original problem as the temperature parameter tends to zero, which justifies that the learned solution in the exploratory formulation can effectively approximate the optimal singular control in the original application without entropy.

It is also worth mentioning that the study of time-inconsistent singular control problems is relatively under-explored in the literature, and there are some emerging studies until recently. \cite{liang2024equilibria} introduced the notion of singular control law to define and solve the equilibrium solution of an irreversible reinsurance problem where the value of singular control appears both in the dynamic and in the objective function. \cite{liang2025stackelberg} applied the definition to a time-inconsistent Stackelberg game in reinsurance. \cite{dai2024dynamic} introduces a different notion of equilibrium definition in solving the time-inconsistent portfolio selection under transaction costs. \cite{cao2025equilibrium} exploited a similar definition to solve a new time-inconsistent dividend problem with a mean-variance objective.

The main contributions of the present paper are summarized as follows:

First, we propose a new method for the randomization of singular controls to encourage the exploration in RL. The randomized form involves two stages of choosing singular control laws. First, at each time $t$, the agent decides a singular control that represents the optimal strategy based on the learned knowledge up to the current time. This optimal singular control in the inner problem can be generated by a consistent family of singular control laws. Second, at each time of the outer problem, the agent flips a biased coin to randomly choose to act according to the current optimal singular control from then on, or to wait for a better time in the future. Similar to \cite{dianetti2024exploratory} and \cite{dianetti2025entropy}, we model the decision of the second step by an auxiliary singular control. The overall randomization is equivalent to a switch at a random activation time from the singular control law $(\mathcal{Q},\emptyset)$ that represents a pure waiting strategy, to a non-trivial one.

Second, in the numerical implementation, the convergence issue on the iteration of time-consistent equilibrium policies might cause a trouble as the policy improvement result no longer holds. In general, it might be a challenge to show that the proposed iteration attains a contraction mapping for the fixed point. To avoid such convergence issue, we establish a q-learning theorem to directly learn the equilibrium value function without policy iteration. We design a neural network to parameterize the equilibrium value function customized for the q-learning theorem. In the numerical plots, we not only illustrate the effectiveness of the algorithm via the convergence of parameter iterations but also demonstrate the evident improvement of the learning performance by considering the randomized policy.

Third, we establish the verification theorem to connect the characterization of equilibrium singular control law for the outer problem to an extended HJB variational inequality.  Using the PDE theory for parabolic equations and Schauder's fixed point theorem, we rigorously prove the existence of equilibrium for a non-standard class of time-inconsistent singular control problem where the singular control variable is treated as an extended state variable. Moreover, we show the convergence of equilibrium value functions as entropy regularization vanishes in the context of singular controls, which is also new to the literature.

The rest of the paper is organized as follows. Section \ref{sec:formulation} formulates the original irreversible reinsurance problem and introduces the two-stage entropy regularized problem for randomizing the singular control decision. Section \ref{sec:inner} solves the inner optimal control problem while Section \ref{sec:outer} characterizes the equilibrium for the time-inconsistent outer problem. Section \ref{sec:ACRL} lays the theoretical foundation of policy evaluation, policy iteration and q-learning for our irreversible reinsurance problem.  In Section \ref{sec:num}, we implement the exploratory RL algorithm by using the randomized singular control laws and compare the performance with a benchmark algorithm without randomization. Finally, Section \ref{existence} rigorously establishes the existence of equilibrium for the time-inconsistent outer problem.  

\section{Entropy Regularization for Randomized Singular Control Laws}\label{sec:formulation}

Let $(\Omega,\mathcal{F},\{\mathcal{F}_{t}\}_{t\ge 0},\mathbb{P})$ be a filtered probability space satisfying the usual conditions, on which $B:=\{B_{t}\}_{t\ge 0}$ is a Brownian motion. Let $\{\mathcal{F}^{B}_{t}\}_{t\ge 0}$ be the right-continuous extension of $\{\mathcal{F}_{t}\}_{t\ge 0}$ generated by $B$. Assume that $(\Omega,\mathcal{F},\mathbb{P})$ is large enough to accommodate a uniform random variable $Z\sim Uniform[0,1]$, which is independent of the Brownian motion $B$. We consider the dynamic decision-making of an insurance company. The uncontrolled risk exposure of the company, denoted by $X$, follows the dynamics
\begin{equation*}
\left\{
\begin{array}{l}
dX_{t}=\mu dt+\sigma dB_{t},\ t\in[0,+\infty),\\
X_{0-}=x_{0},
\end{array}
\right.
\end{equation*}
where $x_{0}$ is the initial risk exposure,  the drift $\mu\in\mathbb{R}$ and the volatility $\sigma>0$ are constants but unknown to the decision maker. The insurance company can sign irreversible reinsurance contracts to mitigate the risk exposure. Using a singular control $\xi$ to model the accumulated transferred risk exposure, we consider the controlled risk exposure $X^{\xi}$ satisfying
\begin{equation}
\label{dyna}
\left\{
\begin{array}{l}
dX^{\xi}_{t}=\mu dt+\sigma dB_{t}-d\xi_{t},\ t\in[0,+\infty),\\
X^{\xi}_{0-}=x_{0}.
\end{array}
\right.
\end{equation}
The insurance company has to balance the running cost from high levels of risk exposure and the cost of reinsurance. The objective for the insurance company is to minimize the following functional over all admissible reinsurance strategies of $\xi$, for $(x,t,y)\in\mathcal{Q}:=\mathbb{R}\times[0,+\infty)\times[0,+\infty)$,
\begin{equation}
\label{object}
\mathbb{E}_{x,t,y}\bigg[\int_{t}^{+\infty}e^{-\beta(r-t)}e^{aX^{\xi}_{r}}dr+\int_{t}^{+\infty}e^{-\beta(r-t)}cd\xi_{r}\bigg].
\end{equation}
In the above expression, we consider the exponential running cost with a constant rate $a>0$. The exponential form represents extreme aversion to high levels of risk exposure. The cost for each unit of reinsurance is a known constant $c>0$. The discount rate $\beta>0$ is also known to the decision maker and assumed to be sufficiently large to satisfy $\beta>\mu a+\frac{1}{2}\sigma^{2}a^{2}$. In \cite{liang2025reinforcement}, we develop a framework of reinforcement learning to learn the optimal singular control law without entropy regularization. The current study aims to design an alternative reinforcement learning algorithm by considering randomized singular control laws and a proper entropy regularization. To this end, we introduce an additional decision on top of the singular control law at each time. That is, at each time $r$ with state-time-control triple $(X_{r-},r,\xi_{r-})$, the agent makes the decision in two stages:\\
(i) He chooses a singular control law $\Xi_{r}$ that can be understood as the optimal one based on the information up to time $r$.\\
(ii) Additionally, he flips a biased coin to decide whether to activate and act according to $\Xi_{r}$ or remain inactive to wait for future information.

\begin{remark}
The idea of introducing an additional binary-choice decision is largely inspired by recent studies \cite{dong2024randomized}, \cite{dianetti2024exploratory}, \cite{dianetti2025entropy} on the randomization of stopping time in reinforcement learning. The intuition is that decisions in both stopping problems and singular control problems are irreversible. 
\end{remark}

This approach leads to a randomized singular control, with each trajectory being the trajectory of a singular control generated by some singular control law after some random stopping time. Mathematically speaking, we need to solve an optimization problem over a pair of $(\Xi=\{\Xi_{r}\}_{r\ge 0},\eta=\{\eta_{r}\}_{r\ge 0})$, where $\Xi_r$ is the optimal singular control law at time $r$ that may possibly depend on the state-time-control triple at time $r$, and $\eta$ is another auxiliary singular control with $\eta_{0-}=0$ and $\eta_{r}\le 1$ for any $r\ge 0$, representing the accumulated activated fraction of the population. Equivalently, under the strategy $\eta$, the agent chooses a random activation time $\tau^{\eta}$ after which she acts according to $\Xi_{\tau^{\eta}}$, where $\tau^{\eta}$ is given below. 

\begin{definition}[Activation time of auxiliary singular control]
\label{st}
Given any $\{\mathcal{F}^{B}_{r}\}_{r\ge t}$-adapted singular control $\eta$ defined on $[t,+\infty)$ with $\eta_{r}\le 1$ for any $r\ge t$, we define the activation time $\tau^{\eta}$ by $\tau^{\eta}:=\inf\{r\ge t|\eta_{r}\ge Z\}$, or equivalently
\begin{equation*}
\mathbb{P}(\tau^{\eta}\le r|\mathcal{F}^{B}_{t})=\eta_{r},\quad\forall r\in[t,+\infty).
\end{equation*}
\end{definition}

The pair of state and control processes under the randomized singular control law $(\Xi,\eta)$ is denoted by $(X^{\Xi,\eta},\xi^{\Xi,\eta})=(X^{\Xi,\tau^{\eta}},\xi^{\Xi,\tau^{\eta}})$, where $(X^{\Xi,r},\xi^{\Xi,r})$ is defined for any $r\in[0,+\infty)$ as the solution to the Skorokhod reflection problem:
\begin{align*}
&\left\{
\begin{array}{l}
dX^{\Xi,r}_{s}=\mu ds+\sigma dB_{s}-d\xi^{\Xi,r}_{s},s\in[0,+\infty), \ \\
(X^{\Xi,r}_{s},s,\xi^{\Xi,r}_{s})\in\overline{W^{\Xi_{r}}_{s}},\ \ \forall s\in[0,+\infty), \ \\
\xi^{\Xi,r}_{s}=y_{0}+\int_{0}^{s}1_{\big\{s'\ge r,(X^{\Xi,r}_{s'-},s',\xi^{\Xi,r}_{s'-})\notin W^{\Xi}_{s} \big\}}d\xi^{\Xi,r}_{s'},s\in[0,+\infty),\ \\
X^{\Xi,r}_{0}=x_{0}.
\end{array}
\right.
\end{align*}
Note that $\Xi_{r}$ only affects the dynamics after time $r$. Therefore, we only consider the admissible set with $\Xi_{r}$ being a partition of the region $\mathcal{Q}^{r}:=\mathbb{R}\times[r,+\infty)\times[0,+\infty)$.

\begin{remark}
Equivalently, the pair of state and control $(X^{\Xi,\eta},\xi^{\Xi,\eta})$ is generated by a randomized singular control law $\Xi^{\tau^{\eta}}$ given by $W^{\Xi^{\tau^{\eta}}}:=W^{\Xi}\bigcup\{(x,t,y)|t<\tau^{\eta}\}$, or notationally
\begin{align*}
\Xi^{\tau^{\eta}}=\left\{
\begin{array}{l}
(\mathcal{Q},\emptyset),\quad t<\tau^{\eta},\\
\big(W^{\Xi},(W^{\Xi})^{c}\big),\quad t\ge\tau^{\eta}.
\end{array}
\right.
\end{align*}
Therefore, our randomization on the singular control is equivalent to a randomization on the singular control law.
\end{remark}

An entropy regularized objective function related to the original problem (\ref{object}) is then defined in the following form:
\begin{equation}
\begin{aligned}
\label{object-B}
&\min_{(\Xi,\eta) }\mathbb{E}_{x_0,0,y_0}\bigg[\int_{0}^{+\infty}e^{-\beta r}e^{aX^{\Xi,\eta}_{r}}dr+\int_{0}^{+\infty}e^{-\beta r}cd\xi^{\Xi,\eta}_{r}-\lambda\int_{0}^{+\infty}e^{-\beta r}\mathcal{E}(\eta_{r})dr\bigg],
\end{aligned}
\end{equation}
where $\mathcal{E}:[0,1]\rightarrow\mathbb{R}$ stands for the entropy function, and $\lambda$ stands for the temperature parameter. 

Using the tower property, we rewrite the entropy regularized objective by
\begin{equation*}
\min_{(\Xi,\eta)}\mathbb{E}_{x_{0},0,y_{0}}\big\{\int_{0}^{+\infty}I(x_{0},0,y_{0},r;\Xi_{r})d\eta_{r}-\lambda\int_{0}^{+\infty}e^{-\beta r}\mathcal{E}(\eta_{r})dr\big\},
\end{equation*}
where the inner objective function $I$ is defined by
\begin{align*}
I(x,t,y,r;\Xi):=&\mathbb{E}_{x,t,y}\bigg[\int_{t}^{+\infty}e^{-\beta (s-t)}e^{aX^{\Xi,r}_{s}}ds+\int_{t}^{+\infty}e^{-\beta (s-t)}cd\xi^{\Xi,r}_{s}\bigg].
\end{align*}
Now, the entropy regularized problem can be rewritten as a two-stage  optimal control formulation in the  following sense:
\begin{equation*}
\min\limits_{\eta}\mathbb{E}_{x_{0},0,y_{0}}\bigg\{\int_{0}^{+\infty}\min\limits_{\Xi_{r}}I(x_{0},0,y_{0},r;\Xi_{r})d\eta_{r}-\lambda\int_{0}^{+\infty}e^{-\beta r}\mathcal{E}(\eta_{r})dr\bigg\}.
\end{equation*}
Recall that at each time $r$, the inner singular control law $\Xi_{r}$ is a partition of  state-time-control space $\mathcal{Q}^{r}=\mathbb{R}\times[r,+\infty)\times[0,+\infty)$. If this optimal $\Xi^{*}_{r}$ intrinsically depends on the activation time $r$ or depends on the initial state-time-control triple $(x,t,y)$, the issue of time-inconsistency arises in the optimal decision of $\{\Xi_{r}\}_{r\in[0,+\infty)}$. However, we will show in the next section that the inner singular control problem is time-consistent, albeit the outer optimal control problem over $\eta$ is time-inconsistent.

\section{Optimal Singular Control Law in the Inner Problem}\label{sec:inner}

We first solve the inner problem to choose the optimal singular control law $\Xi(x,t,y,r)$, i.e., for fixed activation time $r$ and for each initial status $(x,t,y)$, we find a singular law $\Xi^{*}(x,t,y,r)$ that minimizes $I(x,t,y,r;\Xi_{r})$ over $\Xi_{r}$.  

The optimization of $\{\Xi_{r}\}_{r\ge t}$ can be time-inconsistent. At a future time $r+\Delta r$, the optimal singular control law $\Xi^{*}(x,t,y,r+\Delta r)$ can be inconsistent with $\Xi^{*}(x,t,y,r)$. However, it turns out that there is no time-inconsistency issue because of the time-consistent feature of the optimal singular control law.

For the inner optimal control problem, by using the variational method and verification argument, the optimal singular control law $\Xi^{*}=\Xi^{*}(x,t,y,r)$ with the corresponding value function $U(x,t,y,r):=I(x,t,y,r;\Xi^{*})$ are characterized by
\begin{align*}
\left\{
\begin{array}{ll}
&e^{ax}-\beta U(x,t,y,r)+\mathcal{A}U(x,t,y,r)=0,\quad \forall t\in[0,r),\\
&\min\big\{e^{ax}-\beta U(x,t,y,r)+\mathcal{A}U(x,t,y,r),\ c-U_{x}(x,t,y,r)+U_{y}(x,t,y,r)\big\}=0,\quad \forall t\in[r,+\infty)\\
\end{array} \right.
\end{align*}
with
\begin{equation*}
W^{\Xi^{*}(x,t,y,r)}=\big\{(x',t',y')\in\mathcal{Q}^{r}\big|c-U_{x}(x',t',y',r)+U_{y}(x',t',y',r)>0\big\},
\end{equation*}
where $\mathcal{A}$ is defined by
$
\mathcal{A}\varphi:=\varphi_{t}+\mu\varphi_{x}+\frac{1}{2}\sigma^{2}\varphi_{xx} $. 
Defining $\Phi(x,t,y)$ as the solution to
\begin{align}
&\min\Big\{e^{ax}-\beta \Phi(x,t,y)+\mathcal{A}\Phi(x,t,y),\ c-\Phi_{x}(x,t,y)+\Phi_{y}(x,t,y)\Big\}=0,\quad \forall t\in[0,+\infty).\label{eqphi}
\end{align}
Then, the inner optimal value function $U$ satisfies
\begin{equation}
\label{u}
\left\{
\begin{array}{l}
e^{ax}-\beta U(x,t,y,r)+\mathcal{A}U(x,t,y,r)=0,\quad \forall t\in[0,r),\\
U(x,t,y,r)=\Phi(x,t,y),\quad \forall t\in[r,+\infty),
\end{array}
\right.
\end{equation}
and the optimal  singular control law  $\Xi^{*}(x,t,y,r)=\big(W^{\Xi^{*}(x,t,y,r)},(W^{\Xi^{*}(x,t,y,r)})^{c}\big)$ is given by
\begin{equation*}
W^{\Xi^{*}(x,t,y,r)}=\big\{(x',t',y')\in\mathcal{Q}^{r}\big|c-\Phi_{x}(x',t',y')+\Phi_{y}(x',t',y')>0\big\}.
\end{equation*}
It turns out that $\Xi^{*}(x,t,y,r)$ is in fact independent of all $(x,t,y)$, i.e., the optimal singular control law for the problems starting at different $(x,t,y)$ turns out to be the same one. In other words, the sequence of optimal singular control laws is time-consistent.  Therefore, the time-consistent sequence of $\{\Xi^{*}_{r}=\Xi^{*}(x,t,y,r)\}_{r\ge 0}$ with each $\Xi^{*}_{r}$ as a partition of $\mathcal{Q}^{r}\subset\mathcal{Q}$ can be presented by a singular control law $\Xi^{*}$ partitioning $\mathcal{Q}$ with
\begin{equation}
\label{xi}
W^{\Xi^{*}}:=\big\{(x,t,y)\in\mathcal{Q}\big|c-\Phi_{x}(x,t,y)+\Phi_{y}(x,t,y)>0\big\}.
\end{equation}

We next give the solution $\Phi$ to (\ref{eqphi}), the optimal inner strategy $\Xi^{*}$ of the form (\ref{xi}) and solution $U$ to (\ref{u}) all in explicit form. Observing the independence of the PDE (\ref{eqphi}) on variables $t$ and $y$, we obtain $\Phi(x,t,y)=\Phi(x)$ with
\begin{equation}
\label{phi-explicit}
\Phi(x)=\left\{
\begin{array}{l}
C_{a}e^{ax}+C_{b}e^{bx},x<\hat{x},\\
c(x-\hat{x})+C_{a}e^{a\hat{x}}+C_{b}e^{b\hat{x}},x\ge \hat{x},
\end{array}
\right.
\end{equation}
where
\begin{align*}
&b:=\frac{\sqrt{\mu^{2}+2\beta\sigma^{2}}-\mu}{\sigma^{2}}>a, \ \ \ \hat{x}:=\frac{1}{a}\ln(\frac{bc(\beta-\mu a-\frac{1}{2}\sigma^{2}a^{2})}{ba-a^{2}}),\\
&C_{a}:=\frac{1}{\beta-\mu a-\frac{1}{2}\sigma^{2}a^{2}},\ \
C_{b}:=-\frac{a^{2}}{b^{2}(\beta-\mu a-\frac{1}{2}\sigma^{2}a^{2})}e^{(a-b)\hat{x}}.
\end{align*}
Clearly, the optimal inner strategy $\Xi^{*}$ is explicitly given by
\begin{equation}
\label{xi-explicit}
W^{\Xi^{*}}=\big\{(x,t,y)\in\mathcal{Q}\big|x>\hat{x}\big\}.
\end{equation}
Observing the independence of (\ref{u}) on variable $y$, we have $U(x,t,y,r)=U(x,t,r):=U^{r}(x,t)$ with $U^{r}$ solving
\begin{equation*}
\left\{
\begin{array}{l}
e^{ax}-\beta U^{r}(x,t)+\mathcal{A}U^{r}(x,t)=0,\quad \forall t\in[0,r),\\
U^{r}(x,t)=\Phi(x),\quad \forall t\in[r,\infty),
\end{array}
\right.
\end{equation*}
which can be considered as a terminal value problem in $[0,r]$.  Note that $C_{a}e^{ax}$ solves 
\begin{equation*}
e^{ax}-\beta U^{r}(x,t)+\mathcal{A}U^{r}(x,t)=0
\end{equation*}
as an equation of $U^{r}$. By the Feynman–Kac formula, the solution to 
\begin{equation*}
\left\{
\begin{array}{l}
-\beta U^{r}(x,t)+\mathcal{A}U^{r}(x,t)=0,\forall t\in[0,r),\\
U^{r}(x,r)=\Phi(x)-C_{a}e^{ax},
\end{array}
\right.
\end{equation*}
is given by
\begin{equation}
e^{-\beta(r-t)}\mathbb{E}_{x,t}\big[\Phi(X_{r})-C_{a}e^{aX_{r}}\big],
\end{equation}
where $X_{r}=x+\mu(r-t)+\sigma B_{r-t}$. Hence we obtain
\begin{equation*}
U^{r}(x,t)=C_{a}e^{ax}+e^{-\beta(r-t)}\int_{-\infty}^{\infty}\frac{1}{\sqrt{2\pi\sigma^{2}(r-t)}}e^{-\frac{[y-x-\mu(r-t)]^{2}}{2\sigma^{2}(r-t)}}\big[\Phi(y)-C_{a}e^{ay}\big]dy,\quad \forall t\in[0,r).
\end{equation*}
One can observe that for $t\in[0,r)$,  $U(x,t,r)$ depends on $t,r$ only through $r-t$, i.e., $U(x,t,r)=\Psi(x,r-t)$, where
\begin{equation*}
\Psi(p,q):=C_{a}e^{ap}+e^{-\beta q}\int_{-\infty}^{\infty}G(y,p,q)\big[\Phi(y)-C_{a}e^{ay}\big]dy,\forall (p,q)\in\mathbb{R}\times[0,+\infty)
\end{equation*}
with
\begin{equation*}
G(y,p,q):=\frac{1}{\sqrt{2\pi\sigma^{2}q}}e^{-\frac{[y-p-\mu q]^{2}}{2\sigma^{2}q}}.
\end{equation*}

\section{Characterization of Equilibrium in the Outer Problem}\label{sec:outer}
In this section we tackle the outer problem:
\begin{equation}\label{OP}
\min\limits_{\eta}\mathbb{E}_{x_{0},0,0}\Bigg\{\int_{0}^{+\infty}U(x_{0},0,r)d\eta_{r}-\lambda\int_{0}^{+\infty}e^{-\beta r}\mathcal{E}(\eta_{r})dr\Bigg\},
\end{equation}
where $\mathbb{E}_{x,t,z}$ stands for the conditional expectation given $X_{t-}=x$, $\eta_{t-}=z$. It is easy to see from the analysis in the last section that $\eta_{r}$ represents the fraction of the population that has followed the singular control law $\Xi^{*}$ up to time $r$. We also call $\eta_{r}$ the activated fraction of the population up to time $r$.

At an arbitrary future time $t$, suppose that the activated fraction at $t$ is $z\in[0,1]$, i.e., $\eta_{t-}=z$, then the agent should choose a strategy $\{\eta_{s}\}_{s\in[t,+\infty)}$ to control the remaining $1-z$ fraction of population that have not  activated before time $t$. Therefore, the outer problem at time $t$ is to minimizing over $\eta\in\mathcal{D}_{t,z}$ under the objective functional:
\begin{equation}
\label{en-regu}
J(x,t,z;\eta):=\mathbb{E}_{x,t,z}\bigg\{\int_{t}^{+\infty}U(X_{t},t,r)d\eta_{r}-\lambda\int_{t}^{+\infty}e^{-\beta(r-t)}\mathcal{E}(\eta_{r})dr\bigg\},\forall (x,t,z)\in\mathbb{R}\times[0,+\infty)\times[0,1],
\end{equation}
where $X$ is the uncontrolled state process governed by
\begin{equation*}
\left\{
\begin{array}{l}
dX_{r}=\mu dr+\sigma dB_{r},\forall r\in[t,+\infty),\\
X_{t}=x.
\end{array}
\right.
\end{equation*}
The admissible set $\mathcal{D}_{t,z}$ is given by
\begin{align*}
\mathcal{D}_{t,z}:=\big\{\eta=\{\eta_{r}\}_{r\in[t,+\infty)} \mid &\eta\ \text{is non-decreasing, c\`{a}dl\`{a}g and}\ \{\mathcal{F}^{B}_{r}\}_{r\in[t,+\infty)}\text{-adapted}\\ & \text{with}\ \eta_{t-}=z\ {\rm and}\ \lim\limits_{r\rightarrow+\infty}\eta_{r}=1,\ a.s.\big\}.
\end{align*}
Because the integrand $U$ relies on $X_{t}$ and $t$, the problem is time-inconsistent and the dynamic programming principle is not applicable. The problem is a finite-fuel type time-inconsistent singular control problem where the time-inconsistency is due to the dependence of the initial states. We therefore look for a time-consistent equilibrium singular control law $\eta$.

\begin{definition}[Admissible auxiliary singular control law]
Suppose that $\Upsilon=(\mathcal{W}^{\Upsilon},(\mathcal{W}^{\Upsilon})^{c})$ is a partition of 
$
\mathcal{R}:=\mathbb{R}\times[0,+\infty)\times[0,1] $.
We call $\Upsilon$ an \emph{admissible auxiliary singular control law} if 
\begin{equation}
{\label{ceilingcond}}
(x,t,1)\in\mathcal{W}^{\Upsilon},\ \ \forall (x,t)\in\mathbb{R}\times[0,+\infty)
\end{equation}
and for any $(x,t,z)\in\mathcal{R}$, the Skorohod reflection problem 
\begin{equation}\label{SKRP}
\left\{
\begin{array}{l}
dX_{r}=\mu dr+\sigma dB_{r},\ \forall r\in[t,+\infty),\\
(X_{r},r,\eta^{\Upsilon}_{r})\in\overline{\mathcal{W}^{\Upsilon}},\ \forall r\in[t,+\infty),\\
\eta^{\Upsilon}_{r}=z+\int_{t}^{r}1_{\big \{(X_{r'-},r',\eta^{\Upsilon}_{r'-})\notin\mathcal{W}^{\Upsilon}\big\}}d\eta^{\Upsilon}_{r'}, \ \forall r\in[t,+\infty),\\
X_{t-}=x
\end{array}
\right.
\end{equation}
has a unique strong solution $\eta^{\Upsilon}:=\eta^{x,t,z,\Upsilon}$. We call $\eta^{\Upsilon}:=\eta^{x,t,z,\Upsilon}$ the auxiliary singular control generated by $\Upsilon$ at $(x,t,z)$.
\end{definition}

\begin{remark}
Due to the independence of the state process $X$ on $\eta$, one can easily show that a division $\Upsilon$ is admissible if and only if (\ref{ceilingcond}) holds. Moreover, $\eta^{\Upsilon}$ has the explicit expression
\begin{equation*}
\eta^{x,t,z,\Upsilon}_{r}=z+\sup\limits_{s\in[t,r]}\inf\{a\ge 0\mid(X_{s},s,z+a)\in\overline{\mathcal{W}^{\Upsilon}}\}.
\end{equation*}
\end{remark}

\begin{definition}[Equilibrium auxiliary singular control law]
Suppose that $\hat{\Upsilon}$ is an admissible auxiliary singular control law. We call $\hat{\Upsilon}$ an \emph{equilibrium auxiliary singular control law} and $J(x,t,z;\eta^{x,t,z,\hat{\Upsilon}})$ the corresponding equilibrium value function at $(x,t,z)$ if for any initial $(x,t,z)\in\mathcal{R}$, any $\eta'\in\mathcal{D}_{t,z}$ with the property that there exist a constant $M>0$ such that $\eta_{(t+h)-}-\eta_{t}\le Mh$ for\  $h>0$  sufficiently small, it holds that
    \begin{equation}\label{EASCL}
    \liminf\limits_{h\rightarrow 0}\frac{J(x,t,z;\eta^{h})-J(x,t,z;\eta^{x,t,z,\hat{\Upsilon}})}{h}\ge 0
    \end{equation}
where $\eta^{h}$ is defined by
\begin{equation}
 \label{xih0}
\eta^{h}_{r}=\left\{
\begin{array}{l}
\eta'_{r}, \ r\in[t,t+h),\\
\eta^{X_{(t+h)-},t+h,\eta'_{(t+h)-},\hat{\Upsilon}}_{r}, \ r\in[t+h,+\infty).
\end{array}
\right.
\end{equation}
\end{definition}
\noindent
We have the following verification theorem to characterize the time-consistent equilibrium.
\begin{theorem}[Verification theorem of the outer problem]
\label{verif-outer}
Given function $V(x,t,z)$ and a family of functions $\{f^{p,s}(x,t,z)\}_{(p,s)\in\mathbb{R}\times[0,+\infty)}$, denote $f(x,t,z,p,s):=f^{p,s}(x,t,z)$. Consider the infinitesimal operator 
\begin{equation*}
\mathcal{A}\varphi(x,t,z):=\varphi_{t}(x,t,z)+\mu\varphi_{x}(x,t,z)+\frac{1}{2}\sigma^{2}\varphi_{xx}(x,t,z),
\end{equation*}
and  $\hat{\Upsilon}$ from
\begin{align*}
\mathcal{W}^{\hat{\Upsilon}}:=\{(x,t,z)\in\mathcal{R}|U(x,t,t)+V_{z}(x,t,z)=0\}.
\end{align*}
Then $\hat{\Upsilon}$ is an equilibrium auxiliary singular control law and $V(x,t,z)$ is the corresponding equilibrium value function at $(x,t,z)$ if the following conditions hold:\\
(a) $V,f^{p,s}\in C^{2,1,1}(\mathcal{R}\setminus\partial\mathcal{W}^{\hat{\Upsilon}})$ where $\partial\mathcal{W}^{\hat{\Upsilon}}$ consists of finite number of boundary lines $\{(x,t,z)|x=x_{i}\}$ with $i$ belonging to a finite index set, and $V_{x}$ is absolutely continuous across $\partial\mathcal{W}^{\hat{\Upsilon}}$. $f(x,t,z,p,s)$ is $C^{1}$ in $s$ and $C^{2}$ in $p$.\\
(b) The extended HJB system holds
\begin{align}
&\min\{-\lambda\mathcal{E}(z)+\mathcal{AV}(x,t,z)-\mathcal{A}f(x,t,z,x,t)+(\mathcal{A}f^{x,t})(x,t,z),U(x,t,t)+V_{z}(x,t,z)\}=0,\label{v}\\
&-\lambda e^{-\beta(t-s)}\mathcal{E}(z)+(\mathcal{A}f^{p,s})(x,t,z)=0,\quad\forall (x,t,z)\in\mathcal{W}^{\hat{\Upsilon}},\forall(p, s)\in\mathbb{R}\times[0, t],\label{fw}\\
&U(p,s,t)+f^{p,s}_{z}(x,t,z)=0,\quad\forall (x,t,z)\notin\mathcal{W}^{\hat{\Upsilon}},\quad \forall(p, s)\in\mathbb{R}\times[0, t].\label{fp}
\end{align}
(c) $\hat{\Upsilon}$ is admissible.\\
(d) For $\varphi=V,\ f^{p,s},\ (x,t,z)\mapsto f(x,t,z,x,t)$, $\eta'\in\mathcal{D}_{t,z}$ and any $t<T$,
\begin{equation*}
\mathbb{E}_{x,t,z}\int_{t}^{T}\varphi^{2}_{x}(X_{r},r,\eta'_{r})dr<\infty.
\end{equation*}
(e) For any initial $(x,t,z)\in\mathcal{R}$ and any admissible $\eta'\in\mathcal{D}_{t,z}$, the transversality condition holds that
\begin{align*}
&\lim\limits_{T\rightarrow\infty}\mathbb{E}_{x,t,z}f^{p,s}(X_{T},T,\eta'_{T})=0,\\
&\lim\limits_{T\rightarrow\infty}\mathbb{E}_{x,t,z}[V(X_{T},T,\eta'_{T})-f(X_{T},T,\eta'_{T},X_{T},T)]=0.
\end{align*}
Moreover, $f^{p,s}$ has the probabilistic representation 
\begin{equation}
\label{fps}
f^{p,s}(x,t,z)=\mathbb{E}_{x,t,z}\bigg\{\int_{t}^{+\infty}U(p,s,r)d\hat{\eta}_{r}-\lambda\int_{t}^{+\infty}e^{-\beta(r-s)}\mathcal{E}(\hat{\eta}_{r})dr\bigg\}.
\end{equation}
\end{theorem}
\begin{proof}
For the fixed initial value $(x,t,z)\in\mathcal{R}$, let us denote $\hat{\eta}=\eta^{x,t,z,\hat{\Upsilon}}$.
\vskip 5pt\noindent
{\bf Step 1:} We show that $f$ has the probability representation (\ref{fps}). Applying It\^{o}-Tanaka-Meyer's formula to $f^{p,s}(X_{r},r,\hat{\eta}_{r})$ in $\mathcal{W}^{\hat{\Upsilon}}$, we obtain 
\begin{align*}
f^{p,s}(X_{T-},T,\hat{\eta}_{T-})-f^{p,s}(x,t,z)&=\int_{t}^{T}(\mathcal{A}f^{p,s})(X_{r},r,\hat{\eta}_{r})dr+\int_{t}^{T}\sigma f^{p,s}_{x}(X_{r},r,\hat{\eta_{r}})dB_{r}\\&+\int_{t}^{T}f^{p,s}_{z}(X_{r},r,\hat{\eta}_{r})d\hat{\eta}^{c}_{r}
+\sum\limits_{r\in[t,T)}\int_{0}^{\Delta\hat{\eta}_{r}}f^{p,s}_{z}(X_{r},r,\eta_{r-}+u)du.
\end{align*}
Taking expectations, letting $T\rightarrow+\infty$ and applying (\ref{fw}) and (\ref{fp}), we obtain
\begin{align*}
f^{p,s}(x,t,z)=&\mathbb{E}_{x,t,z}\bigg[-\int_{t}^{+\infty}\lambda e^{-\beta(r-s)}\mathcal{E}(\hat{\eta}_{r})dr+\int_{t}^{+\infty}U(p,s,r)d\hat{\eta}^{c}_{r}+\sum\limits_{r\in[t,+\infty)}\int_{0}^{\Delta\hat{\eta}_{r}}U(p,s,r)du\bigg]\\
=&\mathbb{E}_{x,t,z}\bigg[-\int_{t}^{+\infty}\lambda e^{-\beta(r-s)}\mathcal{E}(\hat{\eta}_{r})dr+\int_{t}^{+\infty}U(p,s,r)d\hat{\eta}_{r}\bigg].
\end{align*}
\vskip 3pt \noindent
{\bf Step 2:} We prove that $V(x,t,z)=f(x,t,z,x,t)=J(x,t,z;\hat{\eta})$. Applying It\^{o}-Tanaka-Meyer's formula to $V(X_{r},r,\hat{\eta}_{r})$ and $f(X_{r},r,\hat{\eta},X_{r},r)$ in $\mathcal{W}^{\hat{\Upsilon}}$, we obtain 
\begin{align}
V(X_{T-},T,\hat{\eta}_{T-})-V(x,t,z)&=\int_{t}^{T}(\mathcal{A}V)(X_{r},r,\hat{\eta}_{r})dr+\int_{t}^{T}\sigma V(X_{r},r,\hat{\eta}_{r})dB_{r}\notag\\
&+\int_{t}^{T}V_{z}(X_{r},r,\hat{\eta}_{r})d\hat{\eta}^{c}_{r}+\sum\limits_{r\in[t,T)}\int_{0}^{\Delta\hat{\eta}_{r}}V_{z}(X_{r},r,\hat{\eta}_{r-}+u)du\label{difv}
\end{align}
and
\begin{align}
f(X_{T-},T,\hat{\eta}_{T-},X_{T-},T)&-f(x,t,z,x,t)
=\int_{t}^{T}\mathcal{A}f(X_{r},r,\hat{\eta}_{r},X_{r},r)dr+\int_{t}^{T}\sigma f(X_{r},r,\hat{\eta}_{r},X_{r},r)dB_{r}\notag\\
&+\int_{t}^{T}f_{z}(X_{r},r,\hat{\eta}_{r},X_{r},r)d\hat{\eta}^{c}_{r}+\sum\limits_{r\in[t,T)}\int_{0}^{\Delta\hat{\eta}_{r}}f_{z}(X_{r},r,\hat{\eta}_{r-}+u,X_{r},r)du.\label{diff}
\end{align}
Comparing (\ref{v}), (\ref{fw}) and (\ref{fp}), we obtain
\begin{align*}
&\mathcal{A}V(x,t,z)=\mathcal{A}f(x,t,z,x,t),\  \ \forall (x,t,z)\in\mathcal{W}^{\hat{\Upsilon}},\\
&V_{z}(x,t,z)=f^{t,z}_{z}(x,t,z),\ \ \forall (x,t,z)\notin\mathcal{W}^{\hat{\Upsilon}}.
\end{align*}
Taking expectations on both sides of (\ref{difv}) and (\ref{diff}), letting $T\rightarrow+\infty$, and plugging in the last two equations, we obtain 
\begin{equation*}
V(x,t,z)=f(x,t,z,x,t).
\end{equation*}
Using (\ref{fps}) yields
\begin{equation*}
V(x,t,z)=f(x,t,z,x,t)=J(x,t,z;\hat{\Upsilon}).
\end{equation*}
\vskip 3pt\noindent
{\bf Step 3:} We show that $\hat{\Upsilon}$ is an equilibrium auxiliary singular control law. For any $(t,z)$ and $\eta'\in\mathcal{D}_{t,z}$, define
\begin{equation*}
f^{\eta'}(x,t,z,p,s)=\mathbb{E}_{x,t,z}\bigg\{\int_{t}^{+\infty}U(p,s,r)d\eta'_{r}-\lambda\int_{t}^{+\infty}e^{-\beta(r-s)}\mathcal{E}(\eta'_{r})dr\bigg\}.
\end{equation*}
Then
\begin{align*}
f^{\eta^{h}}(x,t,z,x,t)=&\mathbb{E}_{x,t,z}\bigg\{f^{\eta^{h}}(X_{(t+h)-},t+h,\eta'_{(t+h)-},x,t)-\int_{t}^{t+h}\lambda e^{-\beta(r-s)}\mathcal{E}(\eta'_{r})dr\\
&+\int_{t}^{t+h}U(x,t,r)d(\eta')^{c}_{r}+\sum\limits_{r\in[t,t+h)}U(x,t,r)\Delta\eta'_{r}\bigg\}.
\end{align*}
Noting that $f^{\eta^{h}}(X_{(t+h)-},t+h,\eta'_{(t+h)-},x,t)=f(X_{(t+h)-},t+h,\eta'_{(t+h)-},x,t)$, and $f_{x}$ is absolutely continuous due to the absolutely continuity of $V_{x}=f_{x}+f_{p}$ and $f$ being $C^{2}$ in $p$, we obtain
\begin{align*}
&J(x,t,z;\eta^{h})-J(x,t,z;\hat{\eta})\\
=&\mathbb{E}_{x,t,z}\bigg\{f(X_{(t+h)-},t+h,\eta'_{(t+h)-},x,t)-f(x,t,z,x,t)-\int_{t}^{t+h}\lambda e^{-\beta(r-s)}\mathcal{E}(\eta'_{r})dr\\
&+\int_{t}^{t+h}U(x,t,r)d(\eta')^{c}_{r}+\sum\limits_{r\in[t,t+h)}U(x,t,r)\Delta\eta'_{r}\bigg\}\\
=&\mathbb{E}_{x,t,z}\bigg\{\int_{t}^{t+h}(\mathcal{A}f^{x,t})(X_{r},r,\eta'_{r})dr+\int_{t}^{t+h}f^{x,t}_{z}(X_{r},r,\eta'_{r})d(\eta')^{c}_{r}+\sum\limits_{r\in[t,t+h)}\int_{0}^{\Delta\eta'_{r}}f_{z}^{x,t}(X_{r},r,\eta'_{r-}+u)du\\
&-\int_{t}^{t+h}\lambda e^{-\beta(r-s)}\mathcal{E}(\eta'_{r})dr+\int_{t}^{t+h}U(x,t,r)d(\eta')^{c}_{r}+\sum\limits_{r\in[t,t+h)}U(x,t,r)\Delta\eta'_{r}\bigg\}.
\end{align*}
Thus
\begin{equation*}
\liminf\limits_{h\rightarrow 0}\frac{J(x,t,z;\eta^{h})-J(x,t,z;\hat{\eta})}{h}\ge 0.
\end{equation*}
\end{proof}

We will establish in section \ref{existence} the existence of an equilibrium in the outer time-inconsistent problem. Here, we will first take this existence as given and utilize it for the RL design in the next section.

Next, the optimal solution to the original problem implies that the whole population should activate at the initial time, i.e., $\eta^{0}_{r}\equiv 1,\ \forall r\ge t$. If the equilibrium auxiliary singular control exists, we are interested in whether the equilibrium value function can converge to the optimal value function of the original problem as $\lambda\rightarrow 0$.

Let us denote the equilibrium auxiliary singular control law for the entropy regularized problem (\ref{en-regu}) with temperature $\lambda>0$ by $\Upsilon^{\lambda}$, the generated equilibrium auxiliary singular control by $\eta^{\lambda}:=\eta^{x,t,z,\Upsilon^{\lambda}}$, and the equilibrium value function by $V^{\lambda}(x,t,z):=J^{\lambda}(x,t,z;\eta^{\lambda}),\forall \lambda>0$, where $J^{\lambda}$ is given by (\ref{en-regu}) with the superscript emphasizing the dependence on $\lambda$. Also denote the value function of $\eta^{0}$ for the problem (\ref{en-regu}) with temperature $0$ by $V^{0}(x,t,z):=J^{0}(x,t,z;\eta^{0})=(1-z)\Phi(x)$.

\begin{proposition}
$V^{0}(x,t,z):=(1-z)\Phi(x)$ is an equilibrium value function for Problem (\ref{en-regu}) with temperature $\lambda=0$. The corresponding equilibrium singular control law is $\Upsilon^{0}:=(\emptyset,\mathcal{R})$, and $\eta^{0}$ is the equilibrium singular control generated by $\Upsilon^{0}$ at $(x,t,z)$.
\end{proposition}
\begin{proof}
It is clear that $\eta^{0}$ is the equilibrium singular control generated by $\Upsilon^{0}$ at $(x,t,z)$, and $\eta^{0}$ is admissible. By the optimality of the inner problem, it holds that $\Psi(x,0)=\Phi(x)\le U(x,t,r)=\Psi(x,r-t)$ for any $r\ge t$, which implies that $\Psi_{q}(x,0)\ge 0$. Then $V^{0}(x,t,z)$ and $f^{0}(x,t,z,p,s):=(1-z)\Psi(p,t-s)$ satisfy all conditions of Theorem \ref{verif-outer}. Therefore, we obtain all the desired results. 
\end{proof}\noindent
Similarly, we have the following result, and the proof is omitted.
\begin{proposition}
If $z=1$ is a maximizer of $\mathcal{E}(z)$ over $z\in[0,1]$, then for Problem (\ref{en-regu}), $\Upsilon^{0}:=(\emptyset,\mathcal{R})$ is an equilibrium auxiliary singular control law and $\eta^{0}$ is an equilibrium auxiliary singular control.
\end{proposition}
\noindent
Furthermore, we have the convergence result for the value function as $\lambda\rightarrow 0$.

\begin{proposition}
If an equilibrium value function  $V^{\lambda}$ exists, then it satisfies
\begin{equation*}
V^{0}(x,t,z)-\frac{\lambda}{\beta}\sup\limits_{z\in[0,1]}\mathcal{E}(z)\le V^{\lambda}(x,t,z)\le V^{0}(x,t,z)-\frac{\lambda}{\beta}\mathcal{E}(1),\quad\forall (x,t,z)\in\mathcal{R},
\end{equation*}
which implies that $\lim\limits_{\lambda\rightarrow 0}V^{\lambda}(x,t,z)=V^{0}(x,t,z)$ in $C^{0}(\mathcal{R})$.
\end{proposition}
\begin{proof}
On the one hand, using the fact that $U(x,t,r)\ge U(x,t,t)$ for any $(x,t,r)$ with $r\ge t$, we obtain
\begin{align*}
V^{\lambda}(x,t,z)=&\mathbb{E}_{x,t,z}\Big\{\int_{t}^{+\infty}U(x,t,r)d\hat{\eta}_{r}-\lambda\int_{t}^{+\infty}e^{-\beta(r-t)}\mathcal{E}(\hat{\eta}_{r})dr\Big\}\\
\ge&(1-z)U(x,t,t)-\lambda\sup\limits_{z\in[0,1]}\mathcal{E}(z)\int_{t}^{+\infty}e^{-\beta(r-t)}dr\\
=&V^{0}(x,t,z)-\frac{\lambda}{\beta}\sup\limits_{z\in[0,1]}\mathcal{E}(z).
\end{align*}
On the other hand, $\eta^{0}_{r}\equiv1 \ (\ \forall r\ge t ) $ is a perturbed strategy of $\hat{\eta}$ by setting the initial value of the perturbation to $1$,  the equilibrium condition implies 
\begin{equation*}
J^{\lambda}(x,t,z;\eta^{0})-V^{\lambda}(x,t,z)\ge 0.
\end{equation*}
Then
\begin{equation*}
V^{0}(x,t,z)-\frac{\lambda}{\beta}\mathcal{E}(1)\ge V^{\lambda}(x,t,z).
\end{equation*}
\end{proof}

\section{Actor-Critic RL based on the Two-Stage Formulation}\label{sec:ACRL}

In this section, we design a type of actor-critic RL algorithm by using the randomized singular control laws under entropy regularization. In each step, the agent first evaluates all functions needed, then iterates the strategy, and finally uses the new strategy to generate an action.

\subsection{Policy evaluation}
Two functions play important roles in our algorithm design. The first one is $\Phi(x)=U(x,t,t)=\Psi(x,0),\forall t$, which represents the optimal value function of the original problem. In each step, when the strategy of the inner problem is $\Xi$, the agent should evaluate $\Phi(x;\bar{x}):=\Phi(x;\Xi_{\bar{x}})$ where
\begin{equation}
\Phi(x;\Xi):=\Phi(x,t,y;\Xi):=\mathbb{E}_{x,t,y}\bigg\{\int_{t}^{+\infty}e^{-\beta(r-t)}e^{aX^{\Xi}_{r}}dr+\int_{t}^{+\infty}e^{-\beta(r-t)}cd\xi^{\Xi}_{r}\bigg\},\forall(x,t,y)\in\mathcal{Q}\label{phi}
\end{equation}
with $(X^{\Xi},\xi^{\Xi}):=(X^{x,t,y,\Xi},\xi^{x,t,y,\Xi})$ being the state and control generated by $\Xi$ at $(x,t,y)$.

The function $\Phi(x;\Xi)$ can be evaluated using the following martingale characterization.
\begin{theorem}
\label{pe-phi}
Let $\Xi$ be an admissible singular control law for the inner problem, and $\Phi(x;\Xi)$ is a given function. Then $\Phi(x;\Xi)$ is the value function of $\Xi$ in the sense of (\ref{phi}) if and only if the process
\begin{equation*}
M^{\Phi}:=\bigg\{M^{\Phi}_{r}:=e^{-\beta(r-t)}\Phi(X^{\Xi}_{r};\Xi)+\int_{t}^{r}e^{-\beta(r'-t)}e^{aX^{\Xi}_{r'}}dr'+\int_{t}^{r}e^{-\beta(r'-t)}cd\xi^{\Xi}_{r'}\bigg\}
\end{equation*}
is an $\{\mathcal{F}_{r}\}_{r\in[t,+\infty)}$-martingale.
\end{theorem}

By the definition of $U$, we have $U(x,t,r)=\Phi(x,t,y;\Xi^{*,r})$ where $\Xi^{*,r}$ is defined by $W^{\Xi^{*,r}}=\{(x,t,y)|t<r\ {\rm or} \ (x,t,y)\in W^{\Xi^{*}}\}$. Therefore, the function $U$ can also be evaluated using Theorem \ref{pe-phi}.

The second one is $V(x,t,z)=f^{x,t}(x,t,z)$, which represents the equilibrium value function of the outer problem. In each step, when the strategy of the outer problem is $\Upsilon$, the agent should evaluate $f^{p,s}(x,t,z;\Upsilon)$, which is defined by
\begin{equation*}
f^{p,s}(x,t,z;\Upsilon):=\mathbb{E}_{x,t,z}\bigg\{\int_{t}^{+\infty}U(p,s,r)d\eta^{\Upsilon}_{r}-\lambda\int_{t}^{+\infty}e^{-\beta(r-s)}\mathcal{E}(\eta^{\Upsilon}_{r})dr\bigg\}
\end{equation*}
with $\eta^{\Upsilon}=\eta^{x,t,z,\Upsilon}$ being the auxiliary singular control generated by $\Upsilon$ at $(x,t,z)$. The martingale characterization of $f^{p,s}(x,t,z;\Upsilon)$ is given in the next result.
\begin{theorem}
Let $\Upsilon$ be an admissible auxiliary singular control law and $f^{p,s}(x,t,z;\Upsilon)$ be a given function. Then $f^{p,s}(x,t,z;\Upsilon)$ is the value function of $\Upsilon$ if and only if the process
\begin{equation*}
M^{f,p,s}:=\Bigg\{M^{f,p,s}_{r}:=f^{p,s}(X_{r},r,\eta^{\Upsilon}_{r};\Upsilon)+\int_{t}^{r}U(p,s,r')d\eta^{\Upsilon}_{r'}-\lambda\int_{t}^{r}e^{-\beta(r'-s)}\mathcal{E}(\eta^{\Upsilon}_{r'})dr'\Bigg\}
\end{equation*}
is an $\{\mathcal{F}_{r}\}_{r\in[t,+\infty)}$-martingale.
\end{theorem}

Observe that we have the explicit expressions for $\Phi(x)$ and  $U(p,s,t)$, which motivate us to consider a shared set of parameters such that $\Phi(x)$ can also be learned through policy evaluation on $U(p,s,t)$. 

\subsection{Policy iteration and q-learning}
Note that we are handling a pair of singular control laws $(\Xi,\Upsilon)$.  For $\Xi$, there exists a policy iteration scheme $\Xi\rightarrow \Xi'$ given by
\begin{equation}
\label{xipie}
W^{\Xi'}:= int\big(\{x|q^{\Phi}_{0}(x;\Xi)\ge 0,q^{\Phi}_{1}(x;\Xi)\le 0\}\big).
\end{equation}
where
\begin{align*}
&q^{\Phi}_{0}(x;\Xi):=c-\Phi'(x;\Xi),\\
&q^{\Phi}_{1}(x;\Xi):=e^{ax}-\beta\Phi(x;\Xi)+\mu \Phi'(x;\Xi)+\frac{1}{2}\sigma^{2}\Phi''(x;\Xi).
\end{align*}
Unlike $\Xi$, which is essentially a one-dimensional boundary, the boundary of $\Upsilon$ is a surface in the state-time-control space, which makes its policy iteration significantly more challenging. The outer-layer equilibrium strategy can be obtained by solving the extended HJB system (\ref{v})$\sim$(\ref{fp}) that involves the unknown parameters $\mu,\sigma$, and the inner function $U$. Therefore, if an explicit or numerical solution to this system is available, the policy iteration for the outer strategy $\Upsilon$ can be directly obtained by iterating $\mu,\sigma$ and the inner function $U$.

In cases where a direct solution is not available, we can construct an iteration $\Upsilon\rightarrow\Upsilon'$ for the outer strategy based on (\ref{v})$\sim$(\ref{fp}) as follows:
\begin{equation*}
\mathcal{W}^{\Upsilon'}: =int\Big(\big\{(x, t, z)\big|-\lambda\mathcal{E}(z)+(\mathcal{A}f^{x,t})(x,t,z;\Upsilon)\le 0, \Phi(x; \Xi')+V_{z}(x,t,z;\Upsilon)\ge 0\big\}\Big).
\end{equation*}
However, the convergence of the above policy iteration for the outer time-inconsistent problem is still a challenging problem. Alternatively, we introduce the following q-learning theorem, relying on some constraints for the parameterization, to directly learn the equilibrium value function for the outer time-inconsistent problem.
\begin{theorem}[Martingale characterization for the equilibrium value function and $q$-functions]
\label{qforequi}
Given functions $\hat{f}(x,t,z,p,s):=\hat{f}^{p,s}(x,t,z)\in C^{2,1,1}(\mathcal{R})$ and continuous functions $\hat{q}_{0}(x,t,z,p,s)$ and $\hat{q}_{1}(x,t,z,p,s)$, suppose that for any $\eta\in\mathcal{D}_{t,z}$, $\eta'\in\mathcal{D}'_{t,z}$, and any $T>t$, there exists $h_{0}>0$ such that the following regularity conditions hold:
\begin{align*}
&\mathbb{E}_{x,t,z}\int_{t}^{T}\big[\hat{f}^{p,s}_{x}(X_{r},r,\eta'_{r})\big]^{2}dr+\mathbb{E}_{x,t,z}\int_{t}^{T}\big[\hat{f}_{x}(X_{r},r,\eta'_{r},X_{r},r)\big]^{2}dr<+\infty,\\
&\mathbb{E}_{x,t,z}\sup\limits_{\substack{r\in(t,t+h_{0})\\u\in[0,\Delta\eta'_{r}]}}\!\big|\mathcal{A}\hat{f}^{x,t}(X_{r},r,\eta'_{r-}\!\!+u)\big|+\mathbb{E}_{x,t,z}\sup\limits_{\substack{r\in(t,t+h_{0})\\u\in[0,\Delta\eta'_{r}]}}\!\big|\hat{f}^{x,t}_{z}(X_{r},r,\eta'_{r-}\!\!+u)\big|<+\infty,
\end{align*}
and the constraint conditions
\begin{align*}
&\lim\limits_{T\rightarrow+\infty}\mathbb{E}_{x,t,z}\hat{f}(X_{T},T,\eta_{T},p,s)=0, \quad \forall (x,t,z)\in\mathcal{R},\ (p,s)\in\mathbb{R}\times[0,t],\\
&\min\big\{\hat{q}_{0}(x,t,z,x,t),\hat{q}_{1}(x,t,z,x,t)\big\}=0,\quad\forall (x,t,z)\in\mathcal{R},\\
&\hat{q}_{0}(x,t,z,p,s)\mathbf{1}_{\{\hat{q}_{1}(x,t,z,x,t)>0\}}=0, \quad\forall (x,t,z)\in\mathcal{R},\ (p,s)\in\mathbb{R}\times[0,t],\\
&\hat{q}_{1}(x,t,z,p,s)\mathbf{1}_{\{\hat{q}_{0}(x,t,z,x,t)>0\}}=0, \quad\forall (x,t,z)\in\mathcal{R},\ (p,s)\in\mathbb{R}\times[0,t],
\end{align*}
and the outer strategy $\hat{\Upsilon}$ defined by the waiting region $\mathcal{W}^{\hat{\Upsilon}}:=\big\{(x,t,z)\;\big|\;\hat{q}_{0}(x,t,z,p,s)>0\big\}$ is admissible. Then the following assertions (1) and (2) are equivalent:\\
(1) $\hat{V}(x,t,z):=\hat{f}(x,t,z,x,t)$ is the equilibrium value function, and
\begin{align}
&\hat{q}_{0}(x,t,z,p,s)=\Psi(p,t-s)+\hat{f}^{p,s}_{z}(x,t,z),\label{outer_q0}\\
&\hat{q}_{1}(x,t,z,p,s)=-\lambda e^{-\beta(t-s)}\mathcal{E}(z)+(\mathcal{A}\hat{f}^{p,s})(x,t,z).\label{outer_q1}
\end{align}
(2) For any admissible $\Upsilon$ and any $(x,t,z)\in\mathcal{R}$ and $(p,s)\in\mathbb{R}\times[0,t]$, the stochastic process $M^{\hat{f},\hat{q}_{0},\hat{q}_{1},p,s,\Upsilon}:=M^{\Upsilon}$ is a martingale with respect to $\{\mathcal{F}^{B}_{r}\}_{r\in[t,+\infty)}$, where
\begin{align*}
M^{\Upsilon}_{r}:=&\hat{f}^{p,s}(X_{r},r,\eta^{\Upsilon}_{r})+\int_{t}^{r}\big[\Psi(p,r'-s)-\hat{q}_{0}(X_{r'},r',\eta^{\Upsilon}_{r'},p,s)\big]d(\eta^{\Upsilon})^{c}_{r'}\\
&+\sum\limits_{r'\in[t,r]}\int_{0}^{\Delta\eta^{\Upsilon}_{r'}}\big[\Psi(p,r'-s)-\hat{q}_{0}(X_{r'},r',\eta^{\Upsilon}_{r'-}+u,p,s)\big]du\\
&+\int_{t}^{r}\big[-\lambda e^{-\beta(r'-s)}\mathcal{E}(\eta^{\Upsilon}_{r'})-\hat{q}_{1}(X_{r'},r',\eta^{\Upsilon}_{r'},p,s)\big]dr',\quad\forall r\in[t,+\infty).
\end{align*}
\end{theorem}
\begin{proof}
The proof is similar to Theorem 3.6 in \cite{liang2025reinforcement}.
\end{proof}

With the martingale property established in the above theorem, one can design $q$-learning algorithms that directly learn the equilibrium value function, the $q$-functions, and the equilibrium strategies by minimizing the squared martingale loss or based on the martingale orthogonality conditions.

\section{RL Algorithms and Numerical Implementations}\label{sec:num}

In this section, we numerically implement the reinforcement learning algorithm in the simulation experiment.

In the first step, we design a uniform parameterization for the value functions to prepare for policy evaluation.
\begin{equation*}
\theta_{1}=a,\quad\theta_{2}=b,\quad\theta_{3}=C_{a}.
\end{equation*}
We do policy evaluation on $U$ for the randomized approach, while on $\Phi$ for the benchmark non-randomized approach. These two functions are parameterized as follows.
\begin{align*}
&\!\Phi^{\theta}(x;\bar{x})\!=\!\left\{
\begin{array}{l}
\theta_{3}e^{\theta_{1}x}+C_{b}(\theta;\bar{x})e^{\theta_{2}x},x\le\bar{x},\\
c(x-\bar{x})+C_{v}(\theta;\bar{x}),x>\bar{x},
\end{array}
\right.,\\
&\!U^{\theta}(x,t,r;\bar{x})\!=\!\left\{
\begin{array}{l}
\theta_{3}e^{\theta_{1}x}+e^{-\beta(r-t)}\int_{-\infty}^{\bar{x}}\frac{1}{\sqrt{2\pi\sigma^{2}(\theta)(r-t)}}e^{-\frac{[y-x-\mu(\theta)(r-t)]^{2}}{2\sigma^{2}(\theta)(r-t)}}C_{b}(\theta;\bar{x})e^{\theta_{2}y}dy\\+e^{-\beta(r-t)}\int_{\bar{x}}^{\infty}\frac{1}{\sqrt{2\pi\sigma^{2}(\theta)(r-t)}}e^{-\frac{[y-x-\mu(\theta)(r-t)]^{2}}{2\sigma^{2}(\theta)(r-t)}}\big[c(y-\bar{x})+C_{v}(\theta;\bar{x})-\theta_{3}e^{\theta_{1}y}\big]dy,t\in[0,r),\\
\left\{
\begin{array}{l}
\theta_{3}e^{\theta_{1}x}+C_{b}(\theta;\bar{x})e^{\theta_{2}x},x\le\bar{x},\\
c(x-\bar{x})+C_{v}(\theta;\bar{x}),x>\bar{x},
\end{array}
\right.,t\in[r,+\infty)
\end{array}
\right.
\end{align*}
where
\begin{align*}
&C_{b}(\theta;\bar{x}):=\frac{c-\theta_{1}\theta_{3}e^{\theta_{1}\bar{x}}}{\theta_{2}e^{\theta_{2}\bar{x}}},\quad
C_{v}(\theta;\bar{x}):=\frac{c+(\theta_{2}-\theta_{1})\theta_{3}e^{\theta_{1}\bar{x}}}{\theta_{2}},\\
&\mu(\theta)=\frac{\beta(\theta_{2}^{2}-\theta_{1}^{2})-\frac{\theta_{2}^{2}}{\theta_{3}}}{\theta_{1}\theta_{2}(\theta_{2}-\theta_{1})},\quad
\sigma(\theta)=\sqrt{\frac{2\big[\frac{\theta_{2}}{\theta_{3}}-\beta(\theta_{2}-\theta_{1})\big]}{\theta_{1}\theta_{2}(\theta_{2}-\theta_{1})}}.
\end{align*}

 \begin{remark}
 It is not appropriate to parameterize $\Phi$ using the expression of the optimizer, which does not depend on the given boundary $\bar{x}$ but is additionally $C^{2}$ at the boundary. The boundary in this parameterization is a fixed point for the policy iteration scheme. It is common to do the policy iteration in an appropriate subset of admissible strategies, and it may result in faster convergence. The parameterization should fit the structure of the value function for this subset of admissible strategies. In an extreme case, the parameterization fits the singleton of the optimal case, and the policy iteration no longer leads to different strategies and is thus not appropriate.
 \end{remark}

 \begin{remark}
 The parameterization for the singular control is different from the regular control in that the parameterized value function also depends on the current boundary $\bar{x}$.
 \end{remark}

For the policy evaluation, we use the offline martingale loss algorithm for both $U$ and $\Phi$. To tackle the infinite-time horizon, we impose a truncation at a sufficiently large terminal time $T$. Noting the fact that 
\begin{align*}
&\lim\limits_{T\rightarrow+\infty}e^{-\beta T}U(X^{\Xi_{\bar{x}},r}_{T},T,r;\bar{x})=0, a.s.,\\
&\lim\limits_{T\rightarrow+\infty}e^{-\beta T}\Phi(X^{\Xi_{\bar{x}}}_{T};\bar{x})=0,a.s..
\end{align*}
We can approximate the above terms to zero and obtain the following policy evaluation scheme to iterate $\theta$ for $\Phi$.
\begin{equation}
\theta\leftarrow \theta+\alpha l(m)\Delta t\sum\limits_{n=0}^{N-1}\Big[-e^{-\beta t_{n}}\Phi^{\theta}(X^{\Xi_{\bar{x}}}_{t_{n}};\bar{x})+\sum\limits_{k=n}^{N-1}\big[H_{k}+c^{c}_{k}\big]\Big]\frac{\partial}{\partial \theta}\Phi^{\theta}(X^{\Xi_{\bar{x}}}_{t_{n}};\bar{x}),\label{PEphi}
\end{equation}
where $\alpha$ is the initial learning rate, $l(\cdot)$ is the learning rate schedule function, $H_{k}$ is the running cost in $(t_{k},t_{k+1})$ with interpretation $H_{k}=\int_{t_{k}}^{t_{k+1}}e^{-\beta r}e^{aX^{\Xi_{\bar{x}}}_{r}}dr$, and $c^{c}_{k}$ is the control cost in $(t_{k},t_{k+1})$ given by $c^{c}_{k}=\int_{t_{k}}^{t_{k+1}-}e^{-\beta r}cd\xi^{\Xi_{\bar{x}}}_{r}$.

The policy evaluation for $U$ is similar but more complex. In the randomized singular control approach, we have an activation time $\tau^{\Upsilon_{\bar{x}}}$ generated by the auxiliary singular control law $\Upsilon_{\bar{x}}$ with $W^{\Upsilon_{\bar{x}}}:=\{(x,z)|z>\Gamma(x;\theta,\bar{x})\}$, where we conjecture $\Gamma(x;\theta,\bar{x})$ to be the only boundary for the outer problem. In the algorithm, suppose the grids are $t_{n}=T\frac{n}{N},n=0,1,\cdots,N$, then
\begin{equation*}
\eta^{\Upsilon_{\bar{x}}}_{t_{n}}=\sup\limits_{1\le k\le n}\Gamma(X_{t_{k}}; \theta,\bar{x})
\end{equation*}
where $X$ is the simulated uncontrolled state process. And $\tau^{\Upsilon_{\bar{x}}}$ can be determined by
\begin{equation}
\label{tauupsilon}
\tau^{\Upsilon_{\bar{x}}}=\inf\{t_{k},1\le k\le N|\eta^{\Upsilon_{\bar{x}}}_{t_{k}}\ge Z\},\quad Z\sim Uniform[0,1].
\end{equation}
to make the total probability of activation before time $t_{k}$ be $\eta^{\Upsilon_{\bar{x}}}_{t_{k}}$ for any $k$.

If the infimum in (\ref{tauupsilon}) exists, we iterate
\begin{equation}
\theta\leftarrow \theta+\alpha l(m)\Delta t\sum\limits_{n=0}^{N-1}\Big[-e^{-\beta t_{n}}U^{\theta}(X^{\Xi_{\bar{x}},\tau^{\Upsilon_{\bar{x}}}}_{t_{n}},t_{n},\tau^{\Upsilon_{\bar{x}}};\bar{x})+\sum\limits_{k=n}^{N-1}\big[H_{k}+c^{c}_{k}\big]\Big]\frac{\partial}{\partial \theta}U^{\theta}(X^{\Xi_{\bar{x}},\tau^{\Upsilon_{\bar{x}}}}_{t_{n}},t_{n},\tau^{\Upsilon_{\bar{x}}};\bar{x}).\label{PEu1}
\end{equation}

For the case that the infimum in (\ref{tauupsilon}) does not exist, we approximate this case by the case that $\tau^{\Upsilon_{\bar{x}}}=+\infty$. Now the value function is just the uncontrolled value function
\begin{equation*}
U^{\infty}(x,t;\bar{x}):=C_{a}e^{ax},\quad \text{and}\quad
U^{\infty,\theta}(x,t):=\theta_{3}e^{\theta_{1}x}
\end{equation*}
and we iterate 
\begin{equation}
\theta\leftarrow \theta+\alpha l(m)\Delta t\sum\limits_{n=0}^{N-1}\Big[-e^{-\beta t_{n}}U^{\infty,\theta}(X_{t_{n}},t_{n})+\sum\limits_{k=n}^{N-1}H_{k}\Big]\frac{\partial}{\partial \theta}U^{\infty,\theta}(X_{t_{n}},t_{n}).\label{PEu2}
\end{equation}

\begin{algorithm}[H]
\caption{Non-randomized environment simulator}
\label{nonrandsi}
\renewcommand{\algorithmicrequire}{\textbf{INPUT:}}
\begin{algorithmic}
\REQUIRE Grids $\{t_{n}\}_{0\le n\le N}$, initial $(X_{t_{0}-},t_{0},\xi_{t_{0}-})$, and strategy $\Xi_{\bar{x}}$.
\FOR{$n=0,1,\cdots,N-1$}
    \STATE Calculate immediate jump $\Delta\xi_{t_{n}}=(X_{t_{n}-}-\bar{x})^{+}$ and set $\xi_{t_{n}}=\xi_{t_{n}-}+\Delta\xi_{t_{n}}$, $X_{t_{n}}=X_{t_{n}-}-\Delta\xi_{t_{n}}$.
    \STATE Simulate an uncontrolled increase $I_{\Delta t}\sim N(\mu\Delta t,\sigma^{2}\Delta t)$ and set
    \begin{align*}
    &\xi_{t_{n+1}-}=\xi_{t_{n}}+(X_{t_{n}}+I_{\Delta t}-\bar{x})^{+},\\
    &X_{t_{n+1}-}=X_{t_{n}}-(X_{t_{n}}+I_{\Delta t}-\bar{x})^{+}.
    \end{align*}
    
    Calculate the approximated running cost and control cost
    \begin{align*}
    &H_{t_{n}}=e^{-\beta t_{n}}e^{aX_{t_{n}}}\Delta t,\\
    &c^{c}_{t_{n}}=e^{-\beta t_{n}}c(X_{t_{n}}+I_{\Delta t}-\bar{x})^{+}.
    \end{align*}
\ENDFOR
\RETURN The sequences of $(t_{n})$-triples $\{(X_{t_{n}},t_{n},\xi_{t_{n}})\}_{1\le n\le N-1}$, $(t_{n+1}-)$-triples $\{(X_{t_{n+1}-},t_{n+1},\xi_{t_{n+1}-})\}_{1\le n\le N-1}$, running costs $\{H_{t_{n}}\}_{1\le n\le N-1}$ and control costs $\{c^{c}_{t_{n}}\}_{1\le n\le N-1}$.
\end{algorithmic}
\end{algorithm}

\begin{algorithm}[h]
\caption{Randomized environment simulator}
\label{randsi}
\renewcommand{\algorithmicrequire}{\textbf{INPUT:}}
\begin{algorithmic}
\REQUIRE Grids $\{t_{n}\}_{0\le n\le N}$, initial $(X_{t_{0}-},t_{0},\xi_{t_{0}-},\eta_{t_{0}-})$, and strategy $(\Xi_{\bar{x}},\Upsilon_{\bar{x}})$.

\STATE Initialize activating time $\tau=\infty$, sample a random variable $Z\sim Uniform[0,1]$.

\FOR{$n=0,1,\cdots,N-1$}
    \STATE Calculate immediate auxiliary jump $\Delta\eta_{t_{n}}=\big(\Gamma(X_{t_{n-}}; \theta,\bar{x})-\eta_{t_{n-}}\big)^{+}$. Set 
    \begin{equation*}
    \eta_{t_{n+1}-}=\eta_{t_{n}}=\eta_{t_{n}-}+\Delta \eta_{t_{n}}.
    \end{equation*}

    \IF{$\tau=\infty$}
        \IF{$\eta_{t_{n}}>Z$}
            \STATE Set activation time $\tau=t_{n}$.
        \ELSE
            \STATE Set $\xi_{t_{n+1}-}=\xi_{t_{n}}=\xi_{t_{n}-}, X_{t_{n}}=X_{t_{n}-}$. Simulate an uncontrolled increase $I_{\Delta t}\sim N(\mu\Delta t,\sigma^{2}\Delta t)$ in $[t_{n},t_{n+1})$ and set $X_{t_{n+1}-}=X_{t_{n}}+I_{\Delta t}$.
            \STATE Calculate approximate running cost $H_{t_{n}}=e^{-\beta t_{n}}e^{aX_{t_{n}}}\Delta t$ and control cost $c^{c}_{t_{n}}=0$.
        \ENDIF
    \ENDIF
    \IF{$\tau<\infty$}
        \STATE Feed $(X_{t_{n}-},t_{n},\xi_{t_{n}-})$ and $\Xi_{\bar{x}}$ into the environment simulator \ref{nonrandsi} to obtain the $(t_{n})$-triple $(X_{t_{n}},t_{n},\xi_{t_{n}})$, $(t_{n+1}-)$-triple $(X_{t_{n+1}-},t_{n+1},\xi_{t_{n+1}-})$, running cost $H_{t_{n}}$ and control cost $c^{c}_{t_{n}}$.
    \ENDIF
\ENDFOR
\RETURN The activating time $\tau$, the sequences of $(t_{n})$-quadruples $\{(X_{t_{n}},t_{n},\xi_{t_{n}},\eta_{t_{n}})\}_{0\le n\le N-1}$, $(t_{n+1}-)$-quadruples $\{(X_{t_{n+1}-},t_{n+1},\xi_{t_{n+1}-},\eta_{t_{n+1}-})\}_{0\le n\le N-1}$, running costs $\{H_{t_{n}}\}_{0\le n\le N-1}$ and control costs $\{c^{c}_{t_{n}}\}_{0\le n\le N-1}$.
\end{algorithmic}
\end{algorithm}

The environment simulators for generating samples for offline policy evaluation are summarized in Algorithm \ref{nonrandsi} and Algorithm \ref{randsi}. Algorithm \ref{nonrandsi} is for the non-randomized setting, while Algorithm \ref{randsi} is for the randomized setting. Specifically, for each small interval $[t_{n},t_{n+1})$, the non-randomized simulator generates the singular control and state process at $(X_{t_{n}-},t_{n},\xi_{t_{n}-})$ under $\Xi_{\bar{x}}$. On the other hand, the randomized simulator acts in two steps. First, it calculates whether the agent should activate at time $t_{n}$ according to a random variable $Z$. Second, if the agent should activate at $t_{n}$, then the simulator generates the singular control and state process at $(X_{t_{n}-},t_{n},\xi_{t_{n}-})$ under $\Xi_{\bar{x}}$; otherwise, the simulator keeps the control $\xi$ constant and generate the state process according to the uncontrolled dynamics. The running cost and control cost in $(t_{n},t_{n+1})$, which are expressed as integrals in $(t_{n},t_{n+1})$, are approximated by the values of integrands at $t_{n}$ multiplied by $\Delta t$ or $(\xi_{t_{n+1}-}-\xi_{t_{n}})$.

\begin{remark}
For the randomized approach, an alternative choice is to learn $f$ instead of $U$. The iteration by the martingality w.r.t $f$ is given by
\begin{align*}
\theta\leftarrow\theta+\alpha\sum\limits_{n=0}^{N-1}&\Big[-(f^{x_{0},0})^{\theta}(X_{t_{n}},t_{n},\eta_{t_{n}};\bar{x})+\sum\limits_{k=n}^{N-1}U^{\theta}(x_{0},0,\eta_{t_{n}};\bar{x})\Delta\eta_{t_{n}}-\lambda\sum\limits_{k=n}^{N-1}e^{-\beta t_{n}}\mathcal{E}(\eta_{t_{n}})\Delta t\Big]\\&\cdot\Big[\frac{\partial}{\partial\theta}(f^{x_{0},0})^{\theta}(X_{t_{n}},t_{n},\eta_{t_{n}};\bar{x})-\sum\limits_{k=n}^{N-1}\frac{\partial}{\partial\theta}U^{\theta}(x_{0},0,\eta_{t_{n}};\bar{x})\Delta\eta_{t_{n}}\Big].
\end{align*}
However, the above iteration rule requires the parameterization of both $U$ and $f$, which complicates the learning procedure.
\end{remark}

We highlight the following facts regarding the policy evaluation:

(i) In each step, the jump provides no information because $c$ is known. Hence, we only apply PE in each $(t_{n},t_{n+1})$ interval in (\ref{PEphi}) and (\ref{PEu1})$\sim$(\ref{PEu2}). 

(ii) The summation of gradient term in (\ref{PEphi}) and (\ref{PEu1})$\sim$(\ref{PEu2}) over $n$ can be extremely large due to the cumulation effect. We impose a gradient clipping with a bound $M_{gc}$ on the summation term.

(iii) The parameters $\theta$ are assumed by the agent to satisfy conditions $\mu(\theta)>0$, $\sigma(\theta)>0$ and $\beta-\mu(\theta)\theta_{1}-\frac{1}{2}\sigma^{2}(\theta)\theta_{1}^{2}>0$ (The assumption $\mu(\theta)>0$ can be generalized, see Remark \ref{remarkmu}). Equivalently, it is assumed to be  satisfied that
\begin{equation*}
\theta_{1}>0,\quad \theta_{3}>\frac{1}{\beta},\quad \theta_{2}\in\big(\sqrt\frac{\beta}{\beta-\frac{1}{\theta_{3}}}\theta_{1},\frac{\beta}{\beta-\frac{1}{\theta_{3}}}\theta_{1}\big).
\end{equation*}

Note that (\ref{PEu2}) only iterates $\theta_{1}$ and $\theta_{3}$, we impose boundary clipping for $\theta_{1}$ and $\theta_{3}$ and separately iterate $\theta_{2}$, i.e.,
\begin{align}
&\theta_{1}\leftarrow\max\{\theta_{1},\delta_{bc}\},\label{adj01}\\
&\theta_{3}\leftarrow\max\{\theta_{3},\frac{1}{\beta}+\delta_{bc}\},\label{adj02}\\
&\theta_{2}\leftarrow\frac{\sqrt\frac{\beta}{\beta-\frac{1}{\theta_{3}}}+\frac{\beta}{\beta-\frac{1}{\theta_{3}}}}{2}\theta_{1}.\label{adj03}
\end{align}

\begin{remark}
\label{remarkmu}
The relation $\theta_{2}>\sqrt{\frac{\beta}{\beta-\frac{1}{\theta_{3}}}}\theta_{1}$ follows from the assumption that $\mu(\theta)>0$, which is for simplicity and faster convergence. More generally, we may assume the unknown $\mu$ to have a lower bound and derive similar relations. The assumption $\mu>0$ here is just for simplicity. Even if $\mu$ is assumed to have no known lower bound, the agent can use the clipping $\theta_{2}\in(-\infty,\frac{\beta}{\beta-\frac{1}{\theta_{3}}}\theta_{1})$ and subjectively design a rule to replace the updating rule $\theta_{2}\leftarrow\frac{\sqrt\frac{\beta}{\beta-\frac{1}{\theta_{3}}}+\frac{\beta}{\beta-\frac{1}{\theta_{3}}}}{2}\theta_{1}$.
\end{remark}

For the other case, we only do boundary clipping
\begin{align}
&\theta_{1}\leftarrow\max\{\theta_{1},\delta_{bc}\},\label{adj11}\\
&\theta_{3}\leftarrow\max\{\theta_{3},\frac{1}{\beta}+\delta_{bc}\},\label{adj12}\\
&\theta_{2}\leftarrow \min\Big\{\frac{\beta}{\beta-\frac{1}{\theta_{3}}}\theta_{1}-\delta_{bc},\max\big\{\theta_{2},\sqrt\frac{\beta}{\beta-\frac{1}{\theta_{3}}}\theta_{1}+\delta_{bc}\big\}\Big\}.\label{adj13}
\end{align}

The next step is to learn the boundary $\Gamma(\cdot;\theta,\bar{x})$ in the outer problem and implement the policy iteration for the inner problem. We employ a neural network to learn the equilibrium boundary of the outer problem, with the optimization objective being to minimize the squared martingale loss of the martingale in Theorem \ref{qforequi}. Following the proof of Proposition 4.1 in \cite{liang2025reinforcement}, the policy iteration $\bar{x}\rightarrow\bar{x}'$ can be implemented in Algorithm \ref{PIin}.

\begin{algorithm}[H]
\caption{Policy iteration for $\bar{x}$}
\label{PIin}
\renewcommand{\algorithmicrequire}{\textbf{INPUT:}}
\begin{algorithmic}
\REQUIRE Current parameters $\theta$, current boundary $\bar{x}$, boundary iteration rate $\alpha_{pi}$.
\IF{$(\theta_{2}-\theta_{1})\theta_{1}\theta_{3}e^{\theta_{1}\bar{x}}-c\theta_{2}\le 0$}
    \STATE Calculate $\bar{x}'$ as the root of 
    \begin{equation*}
e^{\theta_{1}x}-\beta c(x-\bar{x})-\beta\frac{c+(\theta_{2}-\theta_{1})\theta_{3}e^{\theta_{1}\bar{x}}}{\theta_{2}}+\mu(\theta)c=0
\end{equation*}
w.r.t variable $x$ in $[\bar{x},+\infty)$.
\ELSE
    \STATE Calculate $\bar{x}'$ as the root of \begin{equation*}
c-\theta_{1}\theta_{3}e^{\theta_{1}x}-ce^{\theta_{2}(x-\bar{x})}+\theta_{1}\theta_{3}e^{\theta_{2}x-(\theta_{2}-\theta_{1})\bar{x}}=0
\end{equation*}
w.r.t variable $x$ in $(-\infty,\bar{x})$.
\ENDIF
\STATE Scale the change by letting
\begin{equation*}
\bar{x}'\leftarrow\bar{x}+\alpha_{pi}(\bar{x}'-\bar{x}).
\end{equation*}
\RETURN the updated boundary $\bar{x}'$
\end{algorithmic}
\end{algorithm}

Now we can integrate the policy evaluation, policy iteration, and q-learning to design an overall actor-critic RL algorithms. The benchmark algorithm using non-randomized singular control is given in Algorithm \ref{benchmark0}. Accordingly, the algorithm using randomized singular control is shown in Algorithm \ref{rand0}.

\begin{remark}
It is worth noting that the benchmark Algorithm \ref{benchmark0} is different from that in \cite{liang2025reinforcement}. Algorithm \ref{benchmark0} here is actor-critic that involves both policy evaluation and policy iteration, while the one in \cite{liang2025reinforcement} is an optimal q-learning algorithm without policy iteration. Such a choice of benchmark algorithm is for better comparison with the randomized algorithm, where two algorithms have the same policy iteration but differ in policy evaluation. That is, the benchmark algorithm evaluates via $\Phi$ and the randomized algorithm evaluates via $U$. Also, we highlight that for better comparison, the randomized algorithm directly learns the target value function through the inner problem, and the outer problem only determines the action of the agent, which is different from the standard RL approach to approximate the target value function by the entropy regularized value function.
\end{remark}
 
In our numerical experiment, the true model coefficients are set as $c=1,\ \beta=0.1,\ a=0.1, \mu=0.25,\ \sigma=1$. We choose the truncation time $T=100$, time step $\Delta t=0.02$ and train $M=500$ episodes with initial guess parameters $\theta^{guess}_{1}=0.15, \theta^{guess}_{2}=0.4, \theta^{guess}_{3}=15$, initial guess boundary $\bar{x}^{guess}=-2.5$ and initial state $x_{0}=1$. The initial learning rates are $[0.1,0.1,1]$ for $\theta_{i},i=1,2,3$, and the learning rate schedule functions are all $l(m):=1.01^{-m}$ for $\theta_{i},i=1,2,3$. The boundary iteration rate is $\alpha_{pi}=0.5$, the gradient clipping bounds are $M_{gc}=[1,1,10]$ for the summation of gradient terms for $\theta_{i}, i=1,2,3$, and the boundary clipping threshold is $\delta_{bc}=0.001$.

\begin{algorithm}[h]
\caption{benchmark non-randomized offline actor-critic algorithm}
\label{benchmark0}
\renewcommand{\algorithmicrequire}{\textbf{INPUT:}}
\begin{algorithmic}
\REQUIRE
Initial $x_{0}$, a cut-off terminal time $T>0$, number of episodes $M$, number of mesh grids $N$, proper parameterization $\Phi^{\theta}(\cdot;\cdot)$, initial learning rates $\alpha$ and a learning rate schedule function $l(\cdot)$. Gradient clipping bound $M_{gc}$, boundary clipping threshold $\delta_{gc}$ and iteration rate $\alpha_{pi}$.\\
\STATE Initialize $\theta$ by a guessed vector value, and initialize boundary $\bar{x}$ by a guessed value. Obtain time step size $\Delta t:=\frac{T}{N}$ and $t_{n}:=n\frac{T}{N}$ for $n=0,1,\cdots,N$.
\FOR {episode $m=1,2,\cdots,M$}
\STATE Feed grids $\{t_{n}\}_{0\le n\le N}$, initial $(X_{t_{0}-},t_{0},\xi_{t_{0}-})=(x_{0},0,0)$, and strategy $\Xi_{\bar{x}}$ defined by $W^{\Xi_{\bar{x}}}=\{x|x<\bar{x}\}$ into the non-randomized environment simulator \ref{nonrandsi} and obtain the output data.

\STATE Apply the martingale loss algorithm to update the parameter $\theta$ by (\ref{PEphi}) with gradient clipping in $[-M_{gc},M_{gc}]$. Then apply boundary clipping (\ref{adj11})$\sim$(\ref{adj13}) to adjust the updated parameters.

\STATE Apply policy iteration \ref{PIin} to update $\bar{x}\leftarrow\bar{x}'$.

\ENDFOR
\end{algorithmic}
\end{algorithm}

\begin{algorithm}[h]
\caption{Randomized offline actor-critic algorithm}
\label{rand0}
\renewcommand{\algorithmicrequire}{\textbf{INPUT:}}
\begin{algorithmic}
\REQUIRE
Initial $x_{0}$, a cut-off terminal time $T>0$, number of episodes $M$, number of mesh grids $N$, proper parameterization $U^{\theta}(\cdot,\cdot,\tau;\cdot)$ for any $\tau$ and $U^{\infty,\theta}(\cdot,\cdot;\cdot)$ for $\tau=\infty$, initial learning rates $\alpha$ and a learning rate schedule function $l(\cdot)$. Gradient clipping bound $M_{gc}$, boundary clipping threshold $\delta_{gc}$, iteration rate $\alpha_{pi}$, temperature $\lambda$. For the outer network, $L^{pre}$ epochs for pretraining, $L^{train}$ epochs for outer network training in every $m^{gap}$ episodes of inner learning, \\
\STATE Initialize $\theta$ by a guessed vector value, and initialize boundary $\bar{x}$ by a guessed value. Obtain time step size $\Delta t:=\frac{T}{N}$ and $t_{n}:=n\frac{T}{N}$ for $n=0,1,\cdots,N$. Initialize the outer neural network and pretrain $L^{pre}$ epochs to obtain the outer boundary $\Gamma(\cdot;\theta,\bar{x})$
\FOR {episode $m=1,2,\cdots,M$}
\IF{$m$ is divisible by $m^{gap}$}
    \STATE Train the outer neural network for $L^{train}$ epochs based on the current estimate of $\theta$.
\ENDIF
\STATE Feed grids $\{t_{n}\}_{0\le n\le N}$, initial $(X_{t_{0}-},t_{0},\xi_{t_{0}-},\eta_{t_{0}-})=(x_{0},0,0,0)$, and strategy $(\Xi_{\bar{x}},\Upsilon_{\bar{x}})$ into the randomized environment simulator \ref{randsi} and obtain the output data.

\IF{activation time $\tau<\infty$}
    \STATE Apply the martingale loss algorithm to update $\theta$ by (\ref{PEu1}) with gradient clipping in $[-M_{gc},M_{gc}]$. Apply boundary clipping (\ref{adj11})$\sim$(\ref{adj13}) to adjust the updated parameters.
\ELSE
    \STATE Apply the martingale loss algorithm to update $\theta$ by (\ref{PEu2}) with gradient clipping in $[-M_{gc},M_{gc}]$. Apply boundary clipping (\ref{adj01})$\sim$(\ref{adj03}) to adjust the updated parameters.
\ENDIF

\STATE Apply policy iteration \ref{PIin} to update $\bar{x}\leftarrow\bar{x}'$.
\ENDFOR
\end{algorithmic}
\end{algorithm}

We employ two sub neural networks in the learning of the outer problem. The value function and $q$-function network $(\hat{f},\hat{q}_{0},\hat{q}_{1})$ consists of 3 hidden layers with 64 neurons per layer, while the outer boundary network $\Gamma$ consists of 2 hidden layers with 32 neurons per layer. The constraints on $(\hat{f},\hat{q}_{0},\hat{q}_{1})$ in Theorem \ref{qforequi} are enforced through structural modifications to the network outputs. The training settings are $L^{pre}=500$, $L^{train}=50$ and $m^{gap}=50$.

\begin{figure}[h]
\centering
\includegraphics[width=6in, keepaspectratio]{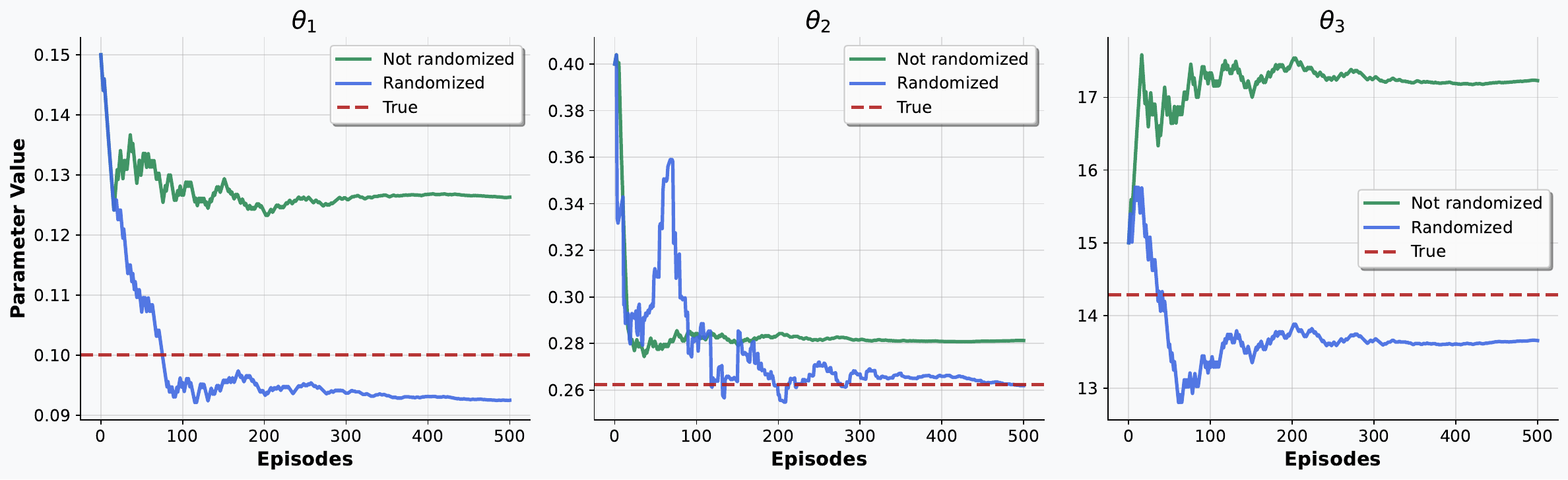}\\
\caption{Comparison of learned parameters $\theta_{i}, i=1,2,3$ of non-randomized Algorithm \ref{benchmark0} and randomized Algorithm \ref{rand0} versus the true parameters during episodes of training.}
\label{fig1}
\end{figure}

\begin{figure}[h]
\centering
\includegraphics[width=6in, keepaspectratio]{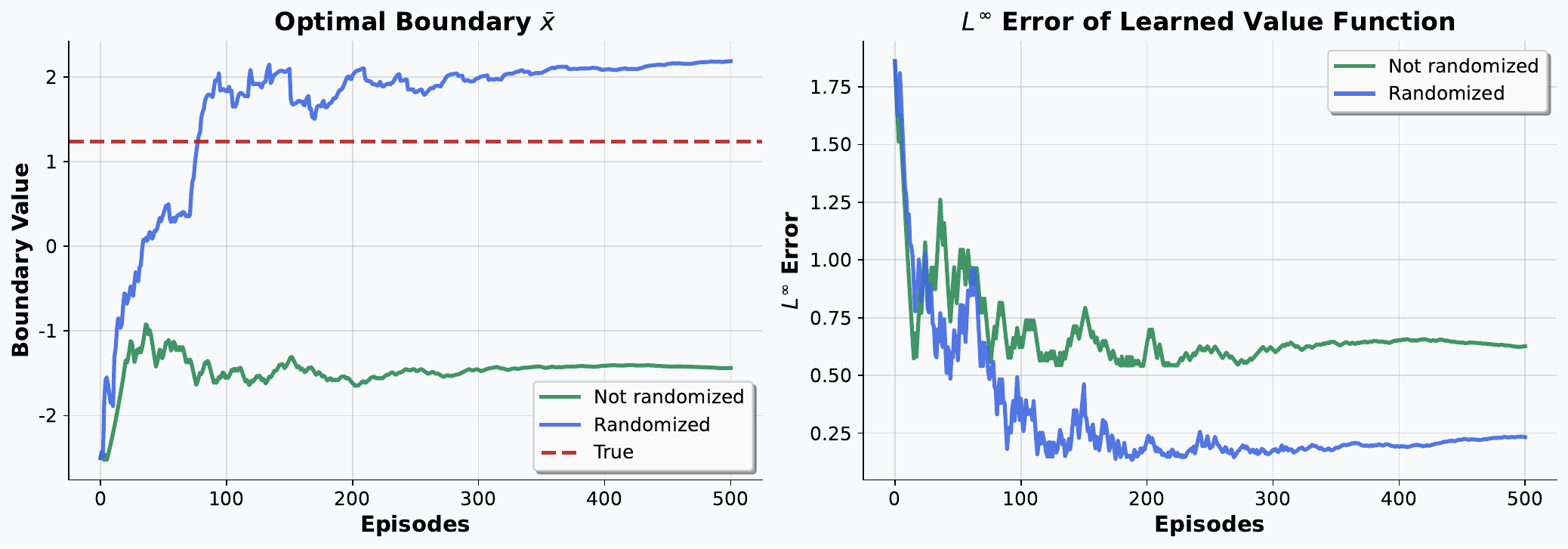}\\
\caption[Short caption for list of figures]{%
Left panel: comparison of learned optimal boundary $\bar{x}$ 
of non-randomized Algorithm \ref{benchmark0} and randomized Algorithm \ref{rand0} 
versus the true optimal boundary during episodes of training.\par
Right panel: comparison of $L^{\infty}$ error of learned value function 
$\Phi(x;\bar{x})$ of non-randomized Algorithm \ref{benchmark0} 
and randomized Algorithm \ref{rand0} during episodes of training.}
\label{fig2}
\end{figure}

The numerical results of Algorithm \ref{benchmark0} and Algorithm \ref{rand0} with temperature parameter $\lambda=0.5$ are shown in Figure \ref{fig1} and Figure \ref{fig2}. Figure \ref{fig1} plots the iteration convergence of learned $\theta_{i}, i=1,2,3$ using the non-randomized Algorithm \ref{benchmark0} and the randomized Algorithm \ref{rand0}, respectively, versus the true values of $\theta_{i}, i=1,2,3$. Figure \ref{fig2} illustrates the learned optimal boundary $\bar{x}$ and the $L^{\infty}$ error of the learned value function $\Phi(x;\bar{x})$ using the non-randomized Algorithm \ref{benchmark0} and the randomized Algorithm \ref{rand0}, respectively. Here, the $L^{\infty}$ error is numerically approximated by the $L^{\infty}$ error of $\Phi(x;\bar{x})$ for $x\in[-100,100]$. Comparing the iteration convergence of parameters and the $L^\infty$ errors, we observe that the randomized Algorithm \ref{rand0} performs better than the non-randomized Algorithm \ref{benchmark0}, particularly, the exploration by considering the randomized singular control laws evidently improves the accuracy of learning in this example.

\section{On the Existence of Equilibrium in the Outer Problem}
\label{existence}
In this section, we turn back to analyzing the extended HJB system (\ref{v})$\sim$(\ref{fp}) and prove the existence of solution. Combing this result and the verification result in Theorem \ref{verif-outer}, we conclude the existence of equilibrium in the time-inconsistent outer problem.

Based on the probability interpretation in Theorem \ref{verif-outer},  \ $V$ satisfying the conditions in Theorem \ref{verif-outer} is homogeneous in time, and the corresponding $f(x,t,z,p,s)$ is a function of $(x,z,p,t-s)$. Moreover, we have the boundary condition at $z=1$ that $f^{p,s}(x,t,1)=-\frac{\lambda}{\beta}e^{-\beta(t-s)}\mathcal{E}(1)$ and $V(x,t,1)=-\frac{\lambda}{\beta}\mathcal{E}(1)$. \\
These results motivate us to start with the existence of solution $G(x,z)$, $\{g^{p,s}(x,t,z)\}_{(p,s)\in\mathbb{R}\times[0,+\infty)}$ to the following HJB equations:
\begin{align}
&\max\big\{-\lambda\mathcal{E}'(z)+\mu G_{x}(x,z)+\frac{1}{2}\sigma^{2}G_{xx}(x,z)-(\mathcal{H}g)(x,z), -\Phi(x)-G(x,z)\big\}=0,\label{G}\\
&-\lambda e^{-\beta(t-s)}\mathcal{E}'(z)+(\mathcal{A}g^{p,s})(x,t,z)=0,\ \forall (x,t,z)\in\mathcal{W}^{\hat{\Upsilon}},\label{gw}\\
&\Psi(p,t-s)+g^{p,s}(x,t,z)=0,\ \forall (x,t,z)\notin\mathcal{W}^{\hat{\Upsilon}},\label{gp}
\end{align}
where
\begin{align*}
&(\mathcal{H}g)(x,z):=[g_{s}+\mu g_{p}+\frac{1}{2}\sigma^{2}(2g_{xp}+g_{pp})](x,0,z,x,0),\\
&\mathcal{W}^{\hat{\Upsilon}}:=\big\{(x,t,z)\in\mathcal{R}\big|\Phi(x)+G(x,z)>0\big\}.
\end{align*}
 $G(x,z)$ and $\{g^{p,s}(x,t,z)\}_{(p,s)\in\mathbb{R}\times[0,+\infty)}$ have the forms:  $G(x,z)=V_{z}(x,t,z), g^{p,s}(x,t,z)=f^{p,s}_{z}(x,t,z)$ for any $(x,t,z,p,s)$. 
\vskip 3pt 
\noindent
The existence of solution to (\ref{G})$\sim$(\ref{gp}) can be studied by fixing $z$. Indeed, for a fixed $z\in[0,1]$, if $\mathcal{E}'(z)\ge 0$, we can verify that $G(x,z)=-\Phi(x)$ and $g^{p,s}(x,t,z)=-\Psi(p,t-s)$ solves the above equation. In the following, we only consider $z$ with $\mathcal{E}'(z)<0$.

\subsection{Preliminary Properties of $\Psi$}
Let us first analyze some basic properties of $\Psi$ to support the proof the existence of solution to the system (\ref{G})$\sim$(\ref{gp}).

\begin{lemma}
\label{property_psi}
$\Psi(p,q)$ is non-decreasing in $q$ and
\begin{equation*}
\Phi(p)=\Psi(p,0)\le\Psi(p,q)\le\lim\limits_{q\rightarrow+\infty}\Psi(p,q)=C_{a}e^{ap},\quad\forall (p,q).
\end{equation*}
\end{lemma}
\begin{proof}
{\bf Step 1.} \ {  Proof of monotonicity and the first inequality: }
\vskip 2pt 
\noindent
Recalling the PDE for $U^{r}(x,t)$, we have the following PDE for $\Psi(x,r-t)=U^{r}(x,t)$:
\begin{equation}\label{pp}
\left\{
\begin{array}{l}
e^{ap}-\beta \Psi(p,q)-\Psi_{q}(p,q)+\mu\Psi_{p}(p,q)+\frac{1}{2}\sigma^{2}\Psi_{pp}(p,q)=0,\quad \forall (p,q)\in\mathbb{R}\times[0,+\infty),\\
\Psi(p,0)=\Phi(p).
\end{array}
\right.
\end{equation}
Using  $U(x,t,r)\ge U(x,t,t),\forall r\ge t$, we have $I(p):=\Psi_{q}(p,0)\ge 0,\forall p\in\mathbb{R}$. Moreover, $\varphi:=\Psi_{q}$ solves
\begin{equation*}
\left\{
\begin{array}{l}
-\beta \varphi(p,q)-\varphi_{q}(p,q)+\mu\varphi_{p}(p,q)+\frac{1}{2}\sigma^{2}\varphi_{pp}(p,q)=0,\quad \forall (p,q)\in\mathbb{R}\times[0,+\infty),\\
\varphi(p,0)=I(p) 
\end{array}
\right.
\end{equation*}
and can be presented by
\begin{equation}\label{ppp}
\varphi(p,q)=e^{-\beta q}\mathbb{E}_{p,0}I(X_{q})
\end{equation}
with $X_{q}=p+\mu q+\sigma B_{q},q\ge 0$. Thus $\varphi(p,q)\ge 0$.
\vskip 2pt \noindent
{\bf Step 2.} { Proof of the second inequality:}\vskip 2pt \noindent
It is sufficient to show that $\lim\limits_{q\rightarrow+\infty}\Psi(p,q)=C_{a}e^{ap}$. Based on the expression of $\Phi$, there exist constants $M_{1}$ and $M_{2}$ such that
\begin{equation*}
0<\Phi(x)\le M_{1}+M_{2}e^{ax},\  \forall x\in\mathbb{R}.
\end{equation*}
It follows  that
\begin{align*}
|\Psi(p,q)-C_{a}e^{ap}|
&=|e^{-\beta q}\int_{-\infty}^{\infty}\frac{1}{\sqrt{2\pi\sigma^{2}q}}e^{-\frac{[y-p-\mu q]^{2}}{2\sigma^{2}q}}\big[\Phi(y)-C_{a}e^{ay}\big]dy|\\
&\le\int_{-\infty}^{+\infty}\frac{1}{\sqrt{2\pi\sigma^{2}q}}e^{-\frac{[y-p-\mu q]^{2}}{2\sigma^{2}q}-\beta q}\big[M_{1}+\tilde{M}_{2}e^{ay}\big]dy\\
&=M_{1} e^{-\beta q} + \tilde{M}_{2} e^{ap} e^{-(\beta-a\mu - \frac{1}{2}a^{2}\sigma^{2} )q},
\end{align*}
which leads to the desired conclusion.
\end{proof}

\begin{lemma}
\label{property_psi_q}
It holds that
\begin{equation*}
0\le\Psi_{q}(p,q)\le e^{ap-\frac{1}{C_{a}}q}+|-e^{a\hat{x}}+\beta c\hat{x}|e^{-\beta q}.
\end{equation*}
\end{lemma}
\begin{proof}
In view of the PDE (\ref{pp}) of $\Psi(p,q)$, sending $q\rightarrow0$, we obtain 
\begin{align}
\Psi_{q}(p,0)=&e^{ap}-\beta\Phi(p)+\mu\Phi'(p)+\frac{1}{2}\sigma^{2}\Phi''(p)\notag\\
=&\begin{cases}
0, & p<\hat{x},\\
e^{ap}-\beta cp-e^{a\hat{x}}+\beta c\hat{x}, & p\ge \hat{x}
\end{cases}\label{eq_Psi}\\
\le& e^{ap}+|-e^{a\hat{x}}+\beta c\hat{x}|.\notag
\end{align}
Then it follows from \eqref{ppp} and \eqref{eq_Psi} that 
\begin{equation*}
\Psi_{q}(p,q)\le e^{-\beta q}\mathbb{E}_{p,0}[e^{aX_{q}}+|-e^{a\hat{x}}+\beta c\hat{x}|]
\le  e^{ap-\frac{1}{C_{a}}q}+|-e^{a\hat{x}}+\beta c\hat{x}|e^{-\beta q}.
\end{equation*}
Note that $\hat{x}$ is the turning point of $x\mapsto e^{ax}-\beta c x$, i.e.,
\begin{equation*}
\hat{x}=\frac{1}{a}\ln(\frac{bc(\beta-\mu a-\frac{1}{2}\sigma^{2}a^{2})}{a(b-a)})=\frac{1}{a}\ln(\frac{bc(\mu+\frac{1}{2}\sigma^{2}(a+b))}{a})> \frac{1}{a}\ln(\frac{c(\mu b+\frac{1}{2}\sigma^{2} b^{2})}{a})=\frac{1}{a}\ln(\frac{\beta c}{a}).
\end{equation*}
Then $I(p):=\Psi_{q}(p,0)$ is non-decreasing and $I(p)\ge 0$ for all $p\in\mathbb{R}$,  hence
\begin{equation*}
\Psi_{q}(p,q)=e^{-\beta q}\mathbb{E}_{p,0}I(X_{q})\ge 0.
\end{equation*}
\end{proof}

\begin{lemma}
\label{property_psi_p_pp}
We have the following bounds for $\Psi_{p}(p,q)$
and $\Psi_{pp}(p,q)$ respectively that
\begin{align*}
&0<aC_{a}e^{ap}(1-e^{-\frac{1}{C_{a}}q})\le \Psi_{p}(p,q)\le aC_{a}e^{ap},\\
&0<a^{2}C_{a}e^{ap}(1-e^{-\frac{1}{C_{a}}q})\le\Psi_{pp}(p,q)\le a^{2}C_{a}e^{ap}.
\end{align*}
\end{lemma}
\begin{proof}
$\varphi(p,q):=\Psi_{p}(p,q)$ satisfies 
\begin{equation*}
\left\{
\begin{array}{l}
ae^{ap}-\beta \varphi(p,q)-\varphi_{q}(p,q)+\mu\varphi_{p}(p,q)+\frac{1}{2}\sigma^{2}\varphi_{pp}(p,q)=0,\quad \forall (p,q)\in\mathbb{R}\times[0,+\infty),\\
\varphi(p,0)=\Phi'(p).
\end{array}
\right.
\end{equation*}
Then
\begin{equation*}
\Psi_{p}(p,q)=aC_{a}e^{ap}+e^{-\beta q}\mathbb{E}_{p,0}[\Phi'(X_{q})-aC_{a}e^{aX_{q}}].
\end{equation*}
Using the fact that $0\le\Phi'(x)\le aC_{a}e^{ax}$, we obtain that
\begin{equation*}
aC_{a}e^{ap}(1-e^{-\frac{1}{C_{a}}q})\le \Psi_{p}(p,q)\le aC_{a}e^{ap}.
\end{equation*}

Now $\phi(p,q):=\Psi_{pp}(p,q)$ satisfies 
\begin{equation*}
\left\{
\begin{array}{l}
a^{2}e^{ap}-\beta \phi(p,q)-\phi_{q}(p,q)+\mu\phi_{p}(p,q)+\frac{1}{2}\sigma^{2}\phi_{pp}(p,q)=0,\quad \forall (p,q)\in\mathbb{R}\times[0,+\infty),\\
\phi(p,0)=\Phi''(p).
\end{array}
\right.
\end{equation*}
Then
\begin{equation*}
\Psi_{pp}(p,q)=a^{2}C_{a}e^{ap}+e^{-\beta q}\mathbb{E}_{p,0}[\Phi''(X_{q})-a^{2}C_{a}e^{aX_{q}}].
\end{equation*}
Note that $\Phi''(\hat{x})=0$ for $x\ge\hat{x}$ and $\Phi''(x)=a^{2}C_{a}e^{ax}(1+\frac{b^{2}C_{b}}{a^{2}C_{a}}e^{(b-a)x})>0$ for $x<\hat{x}$, we have $0\le \Phi''(x)\le a^{2}C_{a}e^{ax}$. Then
\begin{equation*}
a^{2}C_{a}e^{ap}(1-e^{-\frac{1}{C_{a}}q})\le\Psi_{pp}(p,q)\le a^{2}C_{a}e^{ap}.
\end{equation*}
\end{proof}

\begin{proposition}
\label{property_psi_pq}
It holds that $\Psi_{pq}(p,q)\ge 0$.
\end{proposition}
\begin{proof}
$\varphi(p,q):=\Psi_{pq}(p,q)$ satisfies
\begin{equation*}
\begin{cases}
-\beta \varphi(p,q)-\varphi_{q}(p,q)+\mu\varphi_{p}(p,q)+\frac{1}{2}\sigma^{2}\varphi_{pp}(p,q)=0,\quad \forall (p,q)\in\mathbb{R}\times[0,+\infty),\\
\varphi(p,0)=I'(p).
\end{cases}
\end{equation*}
Recalling (\ref{eq_Psi}), we have
\begin{equation*}
I'(p)=\begin{cases}
0, & p<\hat{x},\\
ae^{ap}-\beta c, & p\ge \hat{x}.
\end{cases}
\end{equation*}
In light of  
\begin{equation*}
ae^{a\hat{x}}-\beta c=\frac{bc(\beta-\mu a-\frac{1}{2}\sigma^{2}a^{2})}{b-a}-\beta c
=ac\frac{\beta-b(\mu +\frac{1}{2}\sigma^{2}a)}{b-a}
=ac\frac{b\mu+\frac{1}{2}\sigma^{2}b^{2}-b(\mu +\frac{1}{2}\sigma^{2}a)}{b-a}\ge 0,
\end{equation*}
we deduce that
\begin{equation*}
\Psi_{pq}(p,q)=e^{-\beta q}\mathbb{E}_{p,0}I'(X_{q})\ge 0.
\end{equation*}
\end{proof}

\subsection{Sub-Problem for $\{g^{p,s}(x,t,z)\}_{(p,s)\in\mathbb{R}\times[0,+\infty)}$ given $G$}

In this subsection, for each fixed $z\in[0,1]$ with $\mathcal{E}'(z)<0$ and
 $(p,s)\in\mathbb{R}\times[0,+\infty)$, we focus on the existence of solution $g^{p,s,z}(x,t)$ to the sub-problem
\begin{equation}
\label{eqg}
\begin{aligned}
&-\lambda e^{-\beta(t-s)}\mathcal{E}'(z)+(\mathcal{A}g^{p,s,z})(x,t)=0,\forall (x,t)\in\mathcal{W}^{G}\times[s,+\infty),\\
&\Psi(p,t-s)+g^{p,s,z}(x,t)=0,\forall (x,t,z)\in(\mathcal{W}^{G})^{c}\times[s,+\infty),
\end{aligned}
\end{equation}
where
\begin{equation*}
\mathcal{W}^{G}:=\{x\in\mathbb{R}|G(x)>-\Phi(x)\}
\end{equation*}
denotes the waiting region corresponding to $G$, and $G$ is a given function taken from $\mathcal{R}_{G}$ defined by
\begin{align*}
\mathcal{R}_{G}:=&\{G:\mathbb{R}\rightarrow\mathbb{R}|G\in W^{2}_{p,loc}(\mathbb{R})\bigcap C^{1+\alpha}(\mathbb{R}),\\
&\Vert G(x)\Vert_{W^{2}_{p}([-m,m])}\le C_{G} e^{[a\vee\frac{\mu}{\sigma^{2}}]m},\forall m\ge 1\ \text{for some constant}\ C_{G}\in\mathbb{R},\\
&-\Phi(x)\le G(x)\le -\Phi(x)+\frac{2\sigma^{2}}{\mu^{2}}\tilde{M}_{G}e^{-\frac{\mu}{\sigma^{2}}(x-\tilde{x}_{G})}+1\ \text{for some constant}\ \tilde{M}_{G}, \tilde{x}_{G}\in\mathbb{R}, \\
&(-\infty,\theta_{G})\subseteq\mathcal{W}^{G}\subseteq(-\infty,x_{G})\ \text{for some}\ \theta_{G},x_{G}\in\mathbb{R}\ \text{and } \mathcal{W}^{G}\ \text{consists of finite open intervals}\}.
\end{align*}

\begin{theorem}
There exists a solution $g^{p,s,z}\in  W^{2,1}_{p,loc}((\mathbb{R}\setminus\partial\mathcal{W}^{G})\times[s,+\infty))\bigcap C^{2+\alpha,1+\frac{\alpha}{2}}((\mathbb{R}\setminus\partial\mathcal{W}^{G})\times[s,+\infty))\bigcap C(\mathbb{R}\times[s,+\infty))$ to (\ref{eqg}), where $\partial\mathcal{W}^{G}\subseteq\mathbb{R}$ consists of finite real numbers in $\mathbb{R}$. Moreover, it holds that
\begin{equation}
-C_{a}e^{ap}-\frac{\lambda}{\beta}   e^{-\beta(t-s)}|\mathcal{E}'(z)|\le g^{p,s,z}(x,t)\le -\Phi(p)+\frac{\lambda}{\beta} e^{-\beta(t-s)}|\mathcal{E}'(z)|.\label{esti_g}
\end{equation}
\end{theorem}
\begin{proof}
For $G\in\mathcal{R}_{G}$, $\mathcal{W}^{G}$ consists of an unbounded interval $(-\infty,\Gamma)$ and finite bounded intervals of forms like $(\Gamma_{L},\Gamma_{R})$. We only need to show the existence of $W^{2,1}_{p,loc}\bigcap C^{2+\alpha,1+\frac{\alpha}{2}}$ solution in each of the corresponding sub waiting regions with the Dirichlet boundary condition, and that estimate (\ref{esti_g}) holds.

{\bf 1. Proof in an unbounded interval:}

Suppose $(-\infty,\Gamma)\subseteq\mathcal{W}^{G}$ is an unbounded sub interval and $\mathcal{W}^{s}_{-}:=(-\infty,\Gamma)\times[s,+\infty)$ is the corresponding sub waiting region.

Take $T_{n}\uparrow+\infty$ and define $Q^{n,-}_{T_{n}-s}:=(\Gamma-n,\Gamma)\times(0,T_{n}-s]$, then 
\begin{equation*}
\begin{aligned}
&-\lambda e^{-\beta(T_{n}-t-s)}\mathcal{E}'(z)-h^{p,s,z,n}_{t}(x,t)+\mu h^{p,s,z,n}_{x}(x,t)+\frac{1}{2}\sigma^{2}h^{p,s,z,n}_{xx}(x,t)=0,\forall (x,t)\in Q^{n,-}_{T_{n}-s},\\
&h^{p,s,z,n}(\Gamma,t) =-\Psi(p,T_{n}-t-s),\ h^{p,s,z,n}(\Gamma-n,t) =-\Psi(p,T_{n}-s),\quad \forall t\in(0,T_{n}-s],\\
&h^{p,s,z,n}(x,0) =-\Psi(p,T_{n}-s),\quad\forall x\in(\Gamma-n,\Gamma).
\end{aligned}
\end{equation*}
has a $W^{2,1}_{p}(Q^{n,-}_{T_{n}-s})\bigcap C^{2+\alpha,1+\frac{\alpha}{2}}(Q^{n,-}_{T_{n}-s})\bigcap C(\overline{Q^{n,-}_{T_{n}-s}})$ solution by the $L^{p}$ theory for the first initial boundary value problem, the embedding theorem and Schauder's interior estimate. The comparison principle and Lemma \ref{property_psi} yield that
\begin{equation}
\label{esti_hn}
-C_{a}e^{ap}-\frac{\lambda}{\beta}e^{-\beta(T_{n}-s-t)}|\mathcal{E}'(z)|\le h^{p,s,z,n}(x,t)\le -\Phi(p)+\frac{\lambda}{\beta}e^{-\beta(T_{n}-s-t)}|\mathcal{E}'(z)|.
\end{equation}
Define $g^{p,s,z,n}(x,t)=h^{p,s,z,n}(x,T_{n}-t),\forall(x,t)\in\mathcal{W}^{s}_{n,\frac{1}{n}}:=(\Gamma-n,\Gamma-\frac{1}{n})\times[s,T_{n})$. We have $\mathcal{W}^{s}_{m,\frac{1}{m}}\uparrow\mathcal{W}^{s}_{-}$ as $m\uparrow\infty$. For any $n\ge m+1$, the transformed region of $\mathcal{W}^{s}_{m,\frac{1}{m}}$ under $h^{p,s,z,n}(x,t)=g^{p,s,z,n}(x,T_{n}-t)$ is $Q^{s,m,n}:=(\Gamma-m,\Gamma-\frac{1}{m})\times(T_{n}-T_{m},T_{n}-s]$, and we have 
\begin{equation*}
\Vert g^{p,s,z,n}\Vert_{W^{2,1}_{p}(\mathcal{W}^{s}_{m,\frac{1}{m}})}\le C\Vert h^{p,s,z,n}\Vert_{W^{2,1}_{p}(Q^{s,m,n})}.
\end{equation*}

Interior $L^{p}$ estimate yields that for any $n\ge m+1$
\begin{align*}
&\Vert h^{p,s,z,n}(x,t)\Vert_{W^{2,1}_{p}(Q^{s,m,n})}\\\le& C\bigg[\Vert -\lambda e^{-\beta(T_{n}-t-s)}\mathcal{E}'(z)\Vert_{L^{p}(Q^{s,m+1,n})}+\Vert h^{p,s,z,n}(x,t)\Vert_{L^{p}(Q^{s,m+1,n})}\bigg]\\
\le&C\bigg[\Vert \lambda\mathcal{E}'(z)\Vert_{L^{p}(Q^{s,m+1,n})}+\Vert C_{a}e^{ap}\Vert_{L^{p}(Q^{s,m+1,n})}\bigg],
\end{align*}
where we have used Lemma \ref{property_psi}. Note that the area of $Q^{s,m+1,n}$ is independent of $n$, we deduce that $\Vert h^{p,s,z,n}(x,t)\Vert_{W^{2,1}_{p}(Q^{s,m,n})}$ and thus $\Vert g^{p,s,z,n}\Vert_{W^{2,1}_{p}(\mathcal{W}^{s}_{m,\frac{1}{m}})}$ is bounded by constant independent of $n$.

Then we can choose a subsequence $n_{k}$ such that $g^{p,s,z,n_{k}}$ converges as $n_{k}\rightarrow\infty$ weakly in $W^{2,1}_{p}(\mathcal{W}^{s}_{m,\frac{1}{m}})$ stongly in $C^{1+\alpha,\frac{1+\alpha}{2}}(\mathcal{W}^{s}_{m,\frac{1}{m}})$ for any $m$. The limit $g^{p,s,z}:=\lim\limits_{k\rightarrow\infty}g^{p,s,z,n_{k}}$ is a $W^{2,1}_{p,loc}(\mathcal{W}^{s}_{-})$ $\bigcap C^{1+\alpha,\frac{1+\alpha}{2}}(\mathcal{W}^{s}_{-})\bigcap C(\overline{\mathcal{W}^{s}_{-}})$ solution to (\ref{eqg}). By interior Schauder estimate, $g^{p,s,z}$ is in $C^{2+\alpha,1+\frac{\alpha}{2}}(\mathcal{K})$ for any cylinder $\mathcal{K}\subset \mathcal{W}^{s}_{-}$. Then $g^{p,s,z}\in C^{2+\alpha,1+\frac{\alpha}{2}}(\mathcal{W}^{s}_{-})$.

The estimate (\ref{esti_g}) holds in $\mathcal{W}^{s}_{-}$ due to the estimate (\ref{esti_hn}) of $h^{p,s,z,n}$.

{\bf 2. Proof in a bounded interval:}

Suppose $(\Gamma_{L},\Gamma_{R})\subseteq\mathcal{W}^{G}$ is a bounded sub interval and $\mathcal{W}^{s}:=(\Gamma_{L},\Gamma_{R})\times[s,+\infty)$ is the corresponding sub waiting region.

Take $T_{n}\uparrow+\infty$ and define $Q_{T_{n}-s}:=(\Gamma_{L},\Gamma_{R})\times(0,T_{n}-s]$, then
\begin{equation*}
\begin{aligned}
&-\lambda e^{-\beta(T_{n}-t-s)}\mathcal{E}'(z)-h^{p,s,z,n}_{t}(x,t)+\mu h^{p,s,z,n}_{x}(x,t)+\frac{1}{2}\sigma^{2}h^{p,s,z,n}_{xx}(x,t)=0,\forall (x,t)\in Q_{T_{n}-s},\\
&h^{p,s,z,n}(0,t) =-\Psi(p,T_{n}-s),\ h^{p,s,z,n}(-n,t) =-\Psi(p,T_{n}-s),\quad \forall t\in(0,T_{n}-s],\\
&h^{p,s,z,n}(x,0) =-\Psi(p,T_{n}-s),\quad\forall x\in(\Gamma_{L},\Gamma_{R}).
\end{aligned}
\end{equation*}
has a $W^{2,1}_{p}(Q_{T_{n}-s})\bigcap C^{2+\alpha,1+\frac{\alpha}{2}}(Q_{T_{n}-s})\bigcap C(\overline{Q_{T_{n}-s}})$ solution. The comparison principle still yields (\ref{esti_hn}). Define $g^{p,s,z,n}(x,t)=h^{p,s,z,n}(x,T_{n}-t),\forall(x,t)\in\mathcal{W}^{s}_{\frac{1}{n}}:=(\Gamma_{L}+\frac{1}{n},\Gamma_{R}-\frac{1}{n})\times[s,T_{n})$. We have $\mathcal{W}^{s}_{\frac{1}{m}}\uparrow\mathcal{W}^{s}$ as $m\uparrow\infty$. For any $n\ge m+1$, the transformed region of $\mathcal{W}^{s}_{\frac{1}{m}}$ under $h^{p,s,z,n}(x,t)=g^{p,s,z,n}(x,T_{n}-t)$ is $Q^{s,m,n}:=(\Gamma_{L}+\frac{1}{m},\Gamma_{R}-\frac{1}{m})\times(T_{n}-T_{m},T_{n}-s]$. Following a similar argument as in the unbounded interval case, we have a subsequence $n_{k}$ such that $g^{p,s,z,n_{k}}$ converges as $n_{k}\rightarrow\infty$ weakly in $W^{2,1}_{p}(\mathcal{W}^{s}_{\frac{1}{m}})$ stongly in $C^{1+\alpha,\frac{1+\alpha}{2}}(\mathcal{W}^{s}_{\frac{1}{m}})$ for any $m$. The limit $g^{p,s,z}:=\lim\limits_{k\rightarrow\infty}g^{p,s,z,n_{k}}$ is the desired solution in $\mathcal{W}^{s}$.

\end{proof}

\begin{theorem}
For any $G\in \mathcal{R}_{G}$, the $W^{2,1}_{p,loc}((\mathbb{R}\setminus\partial\mathcal{W}^{G})\times[s,+\infty))\bigcap C^{2+\alpha,1+\frac{\alpha}{2}}((\mathbb{R}\setminus\partial\mathcal{W}^{G})\times[s,+\infty))\bigcap C(\mathbb{R}\times[s,+\infty))$ solution $g^{p,s,z}$ to (\ref{eqg}) with estimate  (\ref{esti_g}) is uniquely given by
\begin{equation}
\label{present_g}
g^{p,s,z}(x,t)=\mathbb{E}_{x,t}[-\Psi(p,\tau-s)-\frac{\lambda}{\beta}(e^{-\beta(t-s)}-e^{-\beta(\tau-s)})\mathcal{E}'(z)],
\end{equation}
where $\tau:=\inf\{r\ge t|X_{r}\notin\mathcal{W}^{G}\}$. As a consequence, $g(x,t,z,p,s)=g^{p,s,z}(x,t)$ is $C^{1}$ in $s$ and $C^{2}$ in $p$ with
\begin{align}
&g^{p,s,z}_{s}(x,t)=\mathbb{E}_{x,t}[\Psi_{q}(p,\tau-s)-\lambda(e^{-\beta(t-s)}-e^{-\beta(\tau-s)})\mathcal{E}'(z)],\label{present_gs}\\
&g^{p,s,z}_{p}(x,t)=\mathbb{E}_{x,t}[-\Psi_{p}(p,\tau-s)],\label{present_gp}\\
&g^{p,s,z}_{pp}(x,t)=\mathbb{E}_{x,t}[-\Psi_{pp}(p,\tau-s)],\label{present_gpp}\\
&g^{p,s,z}_{xp}(x,t)=\mathbb{E}_{x,t}[-\Psi_{p}(p,\tau-s)A(x,\tau-t)].\label{present_gxp}
\end{align}
\end{theorem}
where $\tau:=\inf\{r\ge t|X_{r}\notin \mathcal{W}^{G}\}$ with $X_{r}=x+\mu(r-t)+B_{r-t}$ and 
\begin{equation*}
A(x,r-t)=\frac{\frac{\partial}{\partial x}f_{\tau}(r;x,t)}{f_{\tau}(r;x,t)}
\end{equation*}
with $f_{\tau}(r;x,t)$ being the density function of $\tau$.
\begin{proof}

Let $g^{p,s,z}$ be a solution with the desired regularity and estimate. For fixed $p,s,z$, $(x,t)$ and $\tau<\infty$ a.s., It\^{o}'s formula gives
\begin{equation*}
g^{p,s,z}(X_{\tau},\tau)-g^{p,s,z}(x,t)=\int_{t}^{\tau}\mathcal{A}g^{p,s,z}(X_{r},r)dr+\int_{t}^{r}\sigma g_{x}^{p,s,z}(X_{r},r)dB_{r}
\end{equation*}
That is,
\begin{equation*}
g^{p,s,z}(x,t)=g^{p,s,z}(X_{\tau},\tau)-\int_{t}^{\tau}\lambda e^{-\beta(r-s)}\mathcal{E}(z)dr-\int_{t}^{r}\sigma g_{x}^{p,s,z}(X_{r},r)dB_{r}.
\end{equation*}
The above equality also holds for $\tau_{R,M}:=\min\{\tau,M,\tau_{R}\}$ where $\tau_{R}:=\inf\{r\ge t||X_{r}|\ge R\}$. Taking expectation, we have
\begin{equation*}
g^{p,s,z}(x,t)=\mathbb{E}_{x,t}g^{p,s,z}(X_{\tau_{R,M}},\tau_{R,M})-\mathbb{E}_{x,t}\int_{t}^{\tau_{R,M}}\lambda e^{-\beta(r-s)}\mathcal{E}'(z)dr.
\end{equation*}
Note that $g^{p,s,z}$ is bounded, sending $M\rightarrow+\infty, R\rightarrow+\infty$ and applying the dominated convergence theorem yield (\ref{present_g}). Direct computations lead to (\ref{present_gs})$\sim$(\ref{present_gxp}).
\end{proof}

\begin{corollary}
The function $g$ given by (\ref{present_g}) satisfies $g(x,t,z,p,s)=g(x,t-s,z,p,0)$ with 
\begin{align*}
&-C_{a}e^{ap}\le g^{p,s,z}(x,t)\le -\Phi(p)-\frac{\lambda}{\beta} e^{-\beta(t-s)}\mathcal{E}'(z),\\
&-e^{ap}\le g^{p,s,z}_{s}(x,t)\le -\lambda e^{-\beta(t-s)}\mathcal{E}'(z),\\
&-aC_{a}e^{ap}\le g^{p,s,z}_{p}(x,t)\le 0,\\
&-a^{2}C_{a}e^{ap}\le g^{p,s,z}_{pp}(x,t)\le 0.
\end{align*}
\end{corollary}
\begin{proof}
By (\ref{present_g}), we have $g(x,t,z,p,s)=g(x,t-s,z,p,0)$. The estimates follow directly from (\ref{present_g})$\sim$(\ref{present_gpp}) and Lemmas \ref{property_psi}, \ref{property_psi_q} and  \ref{property_psi_p_pp}.
\end{proof}

To conclude this section, we obtain the well-posedness of the solution mapping $\mathcal{Y}_{g}$ defined by
\begin{equation*}
(\mathcal{Y}_{g}G)(x,t,z,p,s):=\mathbb{E}_{x,t}[-\Psi(p,\tau-s)-\frac{\lambda}{\beta}(e^{-\beta(t-s)}-e^{-\beta(\tau-s)})\mathcal{E}'(z)].
\end{equation*}
We define the image of $\mathcal{Y}_{g}$ by $\mathcal{R}_{g}:=\mathcal{Y}_{g}\mathcal{R}_{G}$.

\subsection{Sub Problem for $G$ given $\{g^{p,s}(x,t,z)\}_{(p,s)\in\mathbb{R}\times[0,+\infty)}$}

In the subsection, we focus on the existence of solution $G$ to the sub problem
\begin{equation}
\label{eqG}
\max\{-\lambda\mathcal{E}'(z)+\mu G^{z}_{x}(x)+\frac{1}{2}\sigma^{2}G^{z}_{xx}(x)-(\mathcal{H}g)(x,z),\ -\Phi(x)-G^{z}(x)\}=0.
\end{equation}
where $g(x,t,z,p,s)=g^{p,s,z}(x,t)$ is a given function taken from $\mathcal{R}_{g}$.

Consider the penalized problem of $G^{z,\epsilon,n}$ on $[-n,n]$:
\begin{equation}
\label{eqG_penal}
\begin{aligned}
&\!-\!\lambda\mathcal{E}'(z)\!+\!\mu G^{z,\epsilon,n}_{x}(x)\!+\!\frac{1}{2}\sigma^{2}G^{z,\epsilon,n}_{xx}(x)\!-\!(\mathcal{H}g)(x,z)\!-\!\alpha_{\epsilon,n}(\Phi(x)\!+\!G^{z,\epsilon,n}(x))=0,\ \forall x\in[-n,n],\\
&G^{z,\epsilon,n}(-n)=-\Phi(-n)+1\\
&G^{z,\epsilon,n}(n)=-\Phi(n).
\end{aligned}
\end{equation}
where the penalty function $\alpha_{\epsilon,n}(\cdot)=C_{n}\alpha_{\epsilon}(\cdot)\in C^{\infty}(\mathbb{R})$ satisfies
\begin{align*}
&\alpha_{\epsilon}(z)\le 0,\ \alpha_{\epsilon}'(z)\ge 0,\quad\forall z\in\mathbb{R},\\
&\lim\limits_{\epsilon\rightarrow0}\alpha_{\epsilon}(z)=0,\ \lim\limits_{\epsilon\rightarrow0}\alpha'_{\epsilon}(z)=0,\quad\forall z>0,\\
&\lim\limits_{\epsilon\rightarrow0}\alpha_{\epsilon}(z)=-\infty,\quad\forall z<0,\\
&\alpha_{\epsilon}(\epsilon)=0,\ \alpha_{\epsilon}(0)=-1,\\
&C_{n}:=-\lambda\mathcal{E}'(z)+\sup\limits_{x\in[-n,n]}\big|(\mathcal{H}g)(x,z)\big|+\sup\limits_{x\in[-n,n]}|\mu\Phi'(x)+\frac{1}{2}\sigma^{2}\Phi''(x)|<+\infty.
\end{align*}

\begin{theorem}
\label{exist_G}
For any $g\in\mathcal{R}_{g}$, there exists a solution $G^{z,\epsilon,n}\in  W^{2}_{p}([-n,n])\bigcap C^{1+\alpha}([-n,n])$ to problem (\ref{eqG_penal}) with
\begin{equation}
-\Phi(x)\le G^{z,\epsilon,n}(x)\le -\Phi(x)+\frac{2\sigma^{2}}{\mu^{2}}\tilde{M}_{g}e^{-\frac{\mu}{\sigma^{2}}(x-\tilde{x}_{g})}+1,\label{esti_penG}
\end{equation}
where $\tilde{M}_{g}$ and $\tilde{x}_{g}$ are constants independent of $\epsilon,n$ but depending on $g$.
\end{theorem}
\begin{proof}
Given $\epsilon$ and $n$, applying the Leray-Schauder fixed point theorem, one gets the existence of $W^{2}_{p}([-n,n])\bigcap C^{1+\alpha}([-n,n])$ solution $G^{z,\epsilon,n}$ to problem (\ref{eqG_penal}). 

For $\underline{G}^{z,\epsilon,n}(x):=-\Phi(x)$, we have
\begin{align*}
&-\lambda\mathcal{E}'(z)+\mu \underline{G}^{z,\epsilon,n}_{x}(x)+\frac{1}{2}\sigma^{2}\underline{G}^{z,\epsilon,n}_{xx}(x)-(\mathcal{H}g)(x,z)-\alpha_{\epsilon,n}(\Phi(x)+\underline{G}^{z,\epsilon,n}(x))\\
=&-\lambda\mathcal{E}'(z)-\mu \Phi'(x)-\frac{1}{2}\sigma^{2}\Phi''(x)-(\mathcal{H}g)(x,z)-\alpha_{\epsilon,n}(0)\ge 0,\forall x\in[-n,n],
\end{align*}
where we used the fact that $\alpha_{\epsilon,n}(0)=-C_{n}$.

For any $\epsilon<1$, for $\overline{G}^{z,\epsilon,n}(x):=-\Phi(x)+A_{1}e^{-A_{2}(x-A_{3})}+1$ with $A_{1},A_{2}>0$ and $A_{3}\in\mathbb{R}$ to be determined, we have
\begin{align}
&-\lambda\mathcal{E}'(z)+\mu \overline{G}^{z,\epsilon,n}_{x}(x)+\frac{1}{2}\sigma^{2}\overline{G}^{z,\epsilon,n}_{xx}(x)-(\mathcal{H}g)(x,z)-\alpha_{\epsilon,n}(\Phi(x)+\overline{G}^{z,\epsilon,n}(x))\notag\\
=&-\lambda\mathcal{E}'(z)+(-\mu A_{1}A_{2}+\frac{1}{2}\sigma^{2}A_{1}A_{2}^{2})e^{-A_{2}(x-A_{3})}-\mu \Phi'(x)-\frac{1}{2}\sigma^{2}\Phi''(x)-(\mathcal{H}g)(x,z),\label{expression_withA}
\end{align}
where we have used the fact that $\alpha_{\epsilon,n}(z)=0$ for $z\ge1$ and $\epsilon<1$. For $g^{p,s,z}\in\mathcal{R}_{g}$, there exists $G_{g}\in\mathcal{R}_{G}$ such that $g=\mathcal{Y}_{g}G_{g}$. As a result, there exists $x_{g}:=x_{G_{g}}\in\mathbb{R}$ such that $x\notin \mathcal{W}^{G_{g}}$ for $x\ge x_{g}$, and
\begin{align*}
(\mathcal{H}g)(x,z)=\begin{cases}
\Psi_{q}(x,0)+\mu\Phi'(x)+\frac{1}{2}\sigma^{2}\Phi''(x), & x\notin\mathcal{W}^{G_{g}},\\
\mathbb{E}_{x,0}[e^{ax}-\beta\Psi(x,\tau)]+\sigma^{2}g_{xp}(x,0,z,x,0), & x\in\mathcal{W}^{G_{g}},
\end{cases}
\end{align*}
which is bounded by some constant $M_{g}$ on $(-\infty,x_{g})$. Then
\begin{align*}
&-\lambda\mathcal{E}'(z)-\mu \Phi'(x)-\frac{1}{2}\sigma^{2}\Phi''(x)-(\mathcal{H}g)(x,z)\\
&\le\begin{cases}
M_{g},&x<x_{g},\\
-\lambda\mathcal{E}'(z)\!-\!I(x),&x\ge x_{g}
\end{cases}
\le\begin{cases}
\tilde{M}_{g},&x<\tilde{x}_{g},\\
0,&x\ge \tilde{x}_{g},
\end{cases}
\end{align*}
where $\tilde{x}_{g}:=I^{-1}(-\lambda\mathcal{E}'(z))\vee x_{g}$ and $\tilde{M}_{g}:=M_{g}-\lambda\mathcal{E}'(z)$. Consequently, (\ref{expression_withA}) is non-positive for
\begin{equation*}
A_{3}=\tilde{x}_{g}>\hat{x},\quad
A_{2}=\frac{\mu}{\sigma^{2}}>0,\quad
A_{1}=\frac{2\sigma^{2}}{\mu^{2}}\tilde{M}_{g}>0.
\end{equation*}
Thanks to the comparison principle, the estimate (\ref{esti_penG}) holds.
\end{proof}

\begin{theorem}
For any $g\in\mathcal{R}_{g}$, there exists a solution $G^{z}\in W^{2}_{p,loc}(\mathbb{R})\bigcap C^{1+\alpha}(\mathbb{R})\bigcap C^{2+\alpha}(\mathbb{R}\setminus\partial \mathcal{W}^{G})$ to (\ref{eqG}), where $\mathcal{W}^{G}:=\{x|G^{z}(x)>-\Phi(x)\}$ is the corresponding waiting region and $\partial \mathcal{W}^{G}$ is the associate boundary. Moreover, it holds that
\begin{equation}
-\Phi(x)\le G^{z}(x)\le -\Phi(x)+\frac{2\sigma^{2}}{\mu^{2}}\tilde{M}_{g}e^{-\frac{\mu}{\sigma^{2}}(x-\tilde{x}_{g})}+1.\label{esti_G}
\end{equation}
\end{theorem}
\begin{proof}
The $W^{2,1}_{p}$ estimate of $G^{z,\epsilon,n}$ in the previous theorem gives
\begin{align*}
&\Vert G^{z,\epsilon,n}(x)\Vert_{W^{2}_{p}([-n,n])}\\
\le& C\Big\{\Vert -\lambda\mathcal{E}'(z)-(\mathcal{H}g)(x,z)-\alpha_{\epsilon,n}(\Phi(x)+G^{z,\epsilon,n}(x))\Vert_{L^{p}([-n,n])}+\Vert-\Phi(x)+\frac{n-x}{2n}\Vert_{W^{2}_{p}([-n,n])}\Big\}\\
\le &C\Big\{\Vert -\lambda\mathcal{E}'(z)-(\mathcal{H}g)(x,z)\Vert_{L^{p}([-n,n])}+\Vert C_{n}\Vert_{L^{p}([-n,n])}+\Vert-\Phi(x)+\frac{n-x}{2n}\Vert_{W^{2}_{p}([-n,n])}\Big\},
\end{align*}
where constant $C$ depends only on $\mu,\sigma,n$.

Then there exists a subsequence $\epsilon_{k}\rightarrow0$ such that $G^{z,\epsilon_{k},n}(x)$ converges weakly in $W^{2}_{p}([-n,n])$ to some $G^{z,n}\in W^{2}_{p}([-n,n])$. We deduce that 
\begin{align*}
&-\lambda\mathcal{E}'(z)+\mu G^{z,n}_{x}(x)+\frac{1}{2}\sigma^{2}G^{z,n}_{xx}(x)-(\mathcal{H}g)(x,z)\\&-\Big[-\lambda\mathcal{E}'(z)-\mu \Phi'(x)-\frac{1}{2}\sigma^{2}\Phi''(x)-(\mathcal{H}g)(x,z)\Big]1_{\{x|G^{z,n}(x)=-\Phi(x)\}}=0,\quad\forall x\in[-n,n],\\
&G^{z,n}(-n)=-\Phi(-n)+1,\\
&G^{z,n}(n)=-\Phi(n).
\end{align*}
For any $m$ and $n\ge m+1$, the interior $L^{p}$ estimate yields that
\begin{equation}
\label{interior_Lp}
\begin{aligned}
&\Vert G^{z,n}(x)\Vert_{W^{2}_{p}([-m,m])}\\
\le& C\bigg[\Big\Vert \lambda\mathcal{E}'(z)\!+\!(\mathcal{H}g)(x,z)\!+\!\big[\!-\!\lambda\mathcal{E}'(z)\!-\!\mu \Phi'(x)\!-\!\frac{1}{2}\sigma^{2}\Phi''(x)\!-\!(\mathcal{H}g)(x,z)\big]1_{\{x|G^{z,n}(x)=-\!\Phi(x)\}}\Big\Vert_{L^{p}([\!-m-1,m+1])}\\
&+\Vert G^{z,n}(x)\Vert_{L^{p}([-m-1,m+1])}\bigg]\\
\le& C\bigg[\Big\Vert \lambda\mathcal{E}'(z)\!+\!(\mathcal{H}g)(x,z)\Big\Vert_{L^{p}([-m-1,m+1])}\!+\!\Big\Vert\!-\!\lambda\mathcal{E}'(z)\!-\!\mu \Phi'(x)\!-\!\frac{1}{2}\sigma^{2}\Phi''(x)\!-\!(\mathcal{H}g)(x,z)\Big\Vert_{L^{p}([\!-m-1,m+1])}\\
&+\Vert -\Phi(x)\Vert_{L^{p}([-m-1,m+1])}+\Vert -\Phi(x)+\frac{2\sigma^{2}}{\mu^{2}}\tilde{M}_{g}e^{-\frac{\mu}{\sigma^{2}}(x-\tilde{x}_{g})}+1\Vert_{L^{p}([-m-1,m+1])}\bigg],
\end{aligned}
\end{equation}
where constant $C$ depends on $\mu,\sigma$ but is independent of $n$. Then $\Vert G^{z,n}\Vert_{W^{2}_{p}([-m,m])},\forall n\ge m+1$ are bounded by constant independent of $n$. Using the standard diagonal argument, we obtain a subsequence $n_{k}\uparrow+\infty$ such that for any $m\ge 1$, the sequence $G^{z,n_{k}}$ converges weakly in $W^{2}_{p}([-m,m])$, strongly in $C^{1+\alpha}([-m,m])$. The limit denoted by $G^{z}$ is a $W^{2}_{p,loc}(\mathbb{R})\bigcap C^{1+\alpha}(\mathbb{R})$ solution to (\ref{eqG}).

The $C^{2+\alpha}(\mathbb{R}\setminus\partial \mathcal{W}^{G})$ regularity follows from the Schauder interior estimate, and the estimate (\ref{esti_G}) is consequent on the estimate (\ref{esti_penG}) of $G^{z,\epsilon,n}$.
\end{proof}

\begin{theorem}
There exists a solution $G^{z}\in W^{2}_{p,loc}(\mathbb{R})\bigcap C^{1+\alpha}(\mathbb{R})$ to (\ref{eqG}) satisfying (\ref{esti_G})$\sim$(\ref{esti_Gnorm}) that admits the representation
\begin{equation}
\label{present_G}
G^{z}(x)=\sup\limits_{\tau}\mathbb{E}_{x,0}\Big[-\Phi(X_{\tau})+\int_{0}^{\tau}\big[-\lambda\mathcal{E}'(z)-(\mathcal{H}g)(X_{r},z)\big]dr\Big].
\end{equation}
Moreover, the waiting region has the form $\mathcal{W}^{G}=\bigcup\limits_{n=1}^{n_{G}-1}(\Gamma^{n}_{L},\Gamma^{n}_{R})\bigcup(-\infty,\Gamma^{n_{G}}_{R})$ where the number of intervals $n_{G}\ge 1$ is finite and $\mathcal{W}^{G}=(-\infty,\Gamma^{1}_{R})$ if $n_{G}=1$. In Addition, it holds that
\begin{equation}
\Vert G(x)\Vert_{W^{2}_{p}([-m,m])}\le C_{G} e^{[a\vee\frac{\mu}{\sigma^{2}}]m},\ \ \forall m\ge 1,\label{esti_Gnorm}
\end{equation}
where $C_{G}$ is a constant independent of $m$ but depending on $G$.
\end{theorem}
\begin{proof}

{\bf Step 1: Construction of the candidate solution:}

Let $G^{z}$ be a $W^{2}_{p,loc}(\mathbb{R})\bigcap C^{1+\alpha}(\mathbb{R})$ solution to (\ref{eqG}) satisfying (\ref{esti_G}) given in the last theorem with its waiting region $\mathcal{W}^{G}$. If there exists some $x_{0}>\Gamma(z)$ such that $x_{0}\in \mathcal{W}^{G}$, define $x_{1}:=\inf\{x>x_{0}|x\notin \mathcal{W}^{G}\}$. Then $x_{1}>\Gamma(z)$ and
\begin{align*}
&G^{z}(x_{1})=-\Phi(x_{1}),\\
&G^{z}_{x}(x_{1})=-\Phi'(x_{1}),\\
&-\lambda\mathcal{E}'(z)+\mu G^{z}_{x}(x_{1})+\frac{1}{2}\sigma^{2}G^{z}_{xx}(x_{1}-)-(\mathcal{H}g)(x_{1}-,z)=0,\\
&-\lambda\mathcal{E}'(z)+\mu G^{z}_{x}(x_{1})+\frac{1}{2}\sigma^{2}G^{z}_{xx}(x_{1}+)-(\mathcal{H}g)(x_{1}+,z)\le0.
\end{align*}
Note that $g(x,t,z,p,s)=-\Psi(p,t-s)$ for $x>\Gamma(z)$, it follows from Lemma \ref{property_psi_pq} that
\begin{equation*}
\frac{\partial}{\partial x}\Big[-\lambda\mathcal{E}'(z)-\mu \Phi'(x)-\frac{1}{2}\sigma^{2}\Phi''(x)-(\mathcal{H}g)(x,z)\Big]=-\Psi_{pq}(x,0)\le 0,\quad \forall x\ge \Gamma(z),
\end{equation*}
which further yields that
\begin{equation*}
\tilde{G}^{z}(x):=\begin{cases}
G^{z}(x),& x\le x_{1},\\
-\Phi(x),& x>x_{1}
\end{cases}
\end{equation*}
is also a $W^{2}_{p,loc}(\mathbb{R})\bigcap C^{1}(\mathbb{R})$ solution to (\ref{eqG}) with its waiting region $\mathcal{W}^{\tilde{G}}\subseteq(-\infty,x_{1})$. In the other case, for any $x_{0}>\Gamma(z)$, we have $x_{0}\notin \mathcal{W}^{G}$. In both cases, we obtain a $W^{2}_{p,loc}(\mathbb{R})\bigcap C^{1}(\mathbb{R})$ solution to (\ref{eqG}), still denoted by $G^{z}$, with the property that $\mathcal{W}^{G}\subseteq(-\infty,x_{G})$ for some $x_{G}\in\mathbb{R}$.

{\bf Step 2: Proof of properties in the waiting region:}

For $x\in(-\infty,\Gamma(z))\bigcap(\mathcal{W}^{G})^{c}$, we have 
\begin{align*}
&-\lambda\mathcal{E}'(z)+\mu G^{z}_{x}(x)+\frac{1}{2}\sigma^{2}G^{z}_{xx}(x)-(\mathcal{H}g)(x,z)\\
=&-\mu \Phi'(x)-\frac{1}{2}\sigma^{2}\Phi''(x)-\mathbb{E}_{x,0}[e^{ax}-\beta\Psi(x,\tau)+\lambda e^{-\beta\tau}\mathcal{E}'(z)]-\sigma^{2}g_{xp}(x,0,z,x,0)\\
\ge&-\lambda\mathcal{E}'(z)\mathbb{E}_{x,0} e^{-\beta\tau}-\sigma^{2}g_{xp}(x,0,z,x,0).
\end{align*}
As $x\rightarrow-\infty$, using (\ref{present_gxp}), we have $g_{xp}(x,0,z,x,0)\sim\frac{1}{\mu}\Psi_{pq}(x,\frac{\Gamma(z)-x}{\mu})\sim\max\{e^{-\beta(\frac{\Gamma(z)-x}{\mu})},e^{ax-\frac{1}{C_{a}}\frac{\Gamma(z)-x}{\mu}}\}$ and $\mathbb{E}_{x,0}e^{-\beta\tau}\sim e^{bx}$. With $b<a+\frac{1}{\mu C_{a}}$ and $b<\frac{\beta}{\mu}$, we conclude that there exists a constant $\theta_{G}$ such that $-\lambda\mathcal{E}'(z)+\mu G^{z}_{x}(x)+\frac{1}{2}\sigma^{2}G^{z}_{xx}(x)-(\mathcal{H}g)(x,z)>0$ for $x\in(-\infty,\theta_{G})$ and thus $(-\infty,\theta_{G})\subseteq\mathcal{W}^{G}\subseteq(-\infty,x_{G})$. 

Note $\mathcal{W}^{G}$ is an open set in $\mathbb{R}$, it must consists of countable open intervals $\bigcap\limits_{n\in\mathcal{S}}(\Gamma^{n}_{L},\Gamma^{n}_{R})$. We now show that the countable index set $\mathcal{S}$ must be finite.

Define $F(x)=G^{z}(x)+\Phi(x)-\frac{\lambda}{\mu}\mathcal{E}'(z)x$, then 
\begin{equation*}
\max\{\mu F'(x)+\frac{1}{2}\sigma^{2}F''(x)+K^{z}(x), -F(x)-\frac{\lambda}{\mu}\mathcal{E}'(z)x\}=0.
\end{equation*}
with
\begin{equation*}
K^{z}(x):=-\mu\Phi'(x)-\frac{1}{2}\sigma^{2}\Phi''(x)-(\mathcal{H}g)(x,z).
\end{equation*}
For $x>\Gamma(z)$, we have $K^{z}(x)=-\Psi_{q}(x,0)=-(e^{ax}-\beta cx-e^{a\hat{x}}+\beta c\hat{x})1_{\{x>\hat{x}\}}$. Then the equation
\begin{equation*}
\mu F'(x)+\frac{1}{2}\sigma^{2}F''(x)+K^{z}(x)=0
\end{equation*}
has a $C^1$ solution
\begin{align*}
F^{0}(x)=&\frac{e^{ax}}{\mu a+\frac{1}{2}\sigma^{2}a^{2}}+(-\frac{\beta c}{2\mu}x^{2}+\frac{\beta c\sigma^{2}-2\mu(e^{a\hat{x}}-\beta c\hat{x})}{2\mu^{2}}x)+\mathbb{E}\int_{0}^{+\infty}[K^{z}(X_{r})+e^{aX_{r}}-\beta cX_{r}-e^{a\hat{x}}+\beta c\hat{x}]dr\\
=&\frac{e^{ax}}{\mu a+\frac{1}{2}\sigma^{2}a^{2}}+(-\frac{\beta c}{2\mu}x^{2}+\frac{\beta c\sigma^{2}-2\mu(e^{a\hat{x}}-\beta c\hat{x})}{2\mu^{2}}x)+\mathbb{E}\int_{0}^{\tau_{\Gamma(z)\vee\hat{x}}}[K^{z}(X_{r})+e^{aX_{r}}-\beta cX_{r}-e^{a\hat{x}}+\beta c\hat{x}]dr,
\end{align*}
where $\tau_{\Gamma(z)\vee\hat{x}}:=\inf\{t\ge 0|X_{t}\ge \Gamma(z)\vee\hat{x}\}$. Define $G^{0}:=F^{0}-\Phi(x)+\frac{\lambda}{\mu}\mathcal{E}'(z)x$, then  $G^{z}$ in any bounded open interval $(\Gamma_{L},\Gamma_{R})\subseteq\mathcal{W}^{G}$ must take the form
\begin{equation*}
G(x)=G^{0}(x)+A_{1}+A_{2}e^{-\frac{\sigma^{2}}{2\mu}x}.
\end{equation*}
The first-order smooth-fit condition gives
\begin{align*}
&G^{0}(\Gamma_{L})+A_{1}+A_{2}e^{-\frac{\sigma^{2}}{2\mu}\Gamma_{L}}=-\Phi(\Gamma_{L}),\\
&G^{0}(\Gamma_{R})+A_{1}+A_{2}e^{-\frac{\sigma^{2}}{2\mu}\Gamma_{R}}=-\Phi(\Gamma_{R}),\\
&(G^{0})'(\Gamma_{L})-A_{2}\frac{\sigma^{2}}{2\mu}e^{-\frac{\sigma^{2}}{2\mu}\Gamma_{L}}=-\Phi'(\Gamma_{L}),\\
&(G^{0})'(\Gamma_{R})-A_{2}\frac{\sigma^{2}}{2\mu}e^{-\frac{\sigma^{2}}{2\mu}\Gamma_{R}}=-\Phi'(\Gamma_{R}).
\end{align*}
Eliminating $A_{1}$ and $A_{2}$, we obtain that
\begin{align*}
&e^{\frac{\sigma^{2}}{2\mu}\Gamma_{L}}[(G^{0})'(\Gamma_{L})+\Phi'(\Gamma_{L})]-e^{\frac{\sigma^{2}}{2\mu}\Gamma_{R}}[(G^{0})'(\Gamma_{R})+\Phi'(\Gamma_{R})]=0,\\
&[(G^{0})(\Gamma_{L})+\Phi(\Gamma_{L})]-[(G^{0})(\Gamma_{R})+\Phi(\Gamma_{R})]+\frac{2\mu}{\sigma^{2}}[(G^{0})'(\Gamma_{L})+\Phi'(\Gamma_{L})]-\frac{2\mu}{\sigma^{2}}[(G^{0})'(\Gamma_{R})+\Phi'(\Gamma_{R})]=0.
\end{align*}
The above two equations for $(\Gamma_{L},\Gamma_{R})$ have no common analytic factor, thus have at most finite solutions on the compact set $[\theta_{G},x_{G}]$. Thus $\mathcal{S}$ is finite.

{\bf Step 3: Proof of probability representation (\ref{present_G}):}

As a by-product of the last step, $G^{z}_{x}(x)$ is of polynomial growth. For any $x\in\mathbb{R}$, define $\tau^{*}:=\inf\{t\ge 0|X_{t}:=x+\mu t+\sigma B_{t}\notin \mathcal{W}^{G}\}$. Then It\^{o}'s formula yields that
\begin{equation}
G^{z}(x)=-\mathbb{E}_{x,0}\Phi(X_{\tau^{*}})+\mathbb{E}_{x,0}\int_{0}^{\tau^{*}}\big[-\lambda\mathcal{E}'(z)-(\mathcal{H}g)(X_{r},z)\big]dr.\label{present_G_opt}
\end{equation}

For any stopping time $\tau$, It\^{o}'s formula gives that 
\begin{align*}
&G^{z}(x)-\mathbb{E}_{x,0}\Big[-\Phi(X_{\tau})+\int_{0}^{\tau}\big[-\lambda\mathcal{E}'(z)-(\mathcal{H}g)(X_{r},z)\big]dr\Big]\\
&\ge G^{z}(x)-\mathbb{E}_{x,0}G^{z}(X_{\tau})-\mathbb{E}_{x,0}\int_{0}^{\tau}\big[-\lambda\mathcal{E}'(z)-(\mathcal{H}g)(X_{r},z)\big]dr\\
&=-\mathbb{E}_{x,0}\int_{0}^{\tau}\big[-\lambda\mathcal{E}'(z)+\mu G^{z}_{x}(X_{r})+\frac{1}{2}\sigma^{2}G^{z}_{xx}(X_{r})-(\mathcal{H}g)(X_{r},z)\big]\ge 0.
\end{align*}
Combining the last formula with (\ref{present_G_opt}) leads to the desired interpretation (\ref{present_G}).

{\bf Step 4: Proof of etimate (\ref{esti_Gnorm}):}

Using (\ref{interior_Lp}) and the fact that $|(\mathcal{H}g)(x,z)|\sim e^{ax}$ as $x\rightarrow+\infty$ and it is bounded as $x\rightarrow-\infty$, we deduce (\ref{esti_Gnorm}).
\end{proof}

In sum, we obtain the well-posedness of the solution mapping $\mathcal{Y}_{G}$ defined by
\begin{equation*}
(\mathcal{Y}_{G}g)(x):=\sup\limits_{\tau}\mathbb{E}_{x,0}\Big[-\Phi(X_{\tau})+\int_{0}^{\tau}[-\lambda\mathcal{E}'(z)-(\mathcal{H}g)(X_{r},z)]dr\Big],
\end{equation*}
and get that $\mathcal{Y}_{G}$ maps any $g\in\mathcal{R}_{g}$ into $\mathcal{R}_{G}$.

\subsection{Fixed Point}

\begin{theorem}
Define $\mathcal{Y}:=\mathcal{Y}_{G}\circ\mathcal{Y}_{g}$, then $\mathcal{Y}$ has a fixed point in $\mathcal{R}_{G}$.
\end{theorem}
\begin{proof}
Take constant $\kappa>a\vee\frac{\mu}{\sigma^{2}}\vee 1$ and define $\Vert G\Vert_{1+\alpha}:=\sum\limits_{m=1}^{\infty}\frac{1}{\kappa^{m}}\Vert G\Vert_{C^{1+\alpha}([-m,m])}$, define $\mathcal{X}:=\{G\in C^{1+\alpha}(\mathbb{R})|\Vert G\Vert_{1+\alpha}<\infty\}$, then $\mathcal{X}$ is a Banach space with norm $\Vert\cdot\Vert_{1+\alpha}$, and $\mathcal{R}_{G}$ is convex and closed in $\mathcal{X}$.

Let $\{G^{(k)}\}_{k\ge 1}$ be an arbitrary sequence in $\mathcal{R}_{G}$. Using Interior $L^{p}$ estimate and the diagonal argument in Theorem \ref{exist_G}, we obtain a subsequence $\{\mathcal{Y}G^{(k_{n})}\}_{n\ge 1}$ that converges weakly in $W^{2}_{p}([-m,m])$ and strongly in $C^{1+\alpha}([-m,m])$ for any $m\ge 1$. The limits on $[-m,m],\forall m\ge 1$, coincide and are denoted by $G^{*}$. Then $G^{*}\in\mathcal{X}$ and $\mathcal{Y}$ is a compact operator from $\mathcal{R}_{G}$ to $\mathcal{X}$.

We then prove the continuity of $\mathcal{Y}$. Let $G^{(k)}\rightarrow G^{*}$ be a converging sequence in $\mathcal{R}_{G}$, i.e. $\Vert G^{(k)}-G^{*}\Vert_{C^{1+\alpha}([-m,m])}\rightarrow 0$ for any $m\ge 1$. Define $g^{(k)}:=\mathcal{Y}_{g}G^{(k)}$, $g^{*}:=\mathcal{Y}_{g}G^{*}$.

Let $G^{(k),\epsilon,n}$ and $G^{*,\epsilon,n}$ be the approximate function in the penalty approximation, then 
\begin{align*}
&\Vert G^{(k)}(x)-G^{*}(x)\Vert_{W^{2}_{p}([-m,m])}\\
\le& C\bigg\{\Vert G^{(k)}(x)-G^{(k),\epsilon,n}(x)\Vert_{W^{2}_{p}([-m,m])}+\Vert G^{*}(x)-G^{*,\epsilon,n}(x)\Vert_{W^{2}_{p}([-m,m])}+\Vert H^{(k)}(x)-H^{*}(x)\Vert_{L^{p}([-m,m])}\\
&+\Vert \alpha^{\epsilon,n}(\Phi(x)+G^{(k),\epsilon,n}(x))-\alpha^{\epsilon,n}(\Phi(x)+G^{*,\epsilon,n}(x))\Vert_{L^{p}([-m,m])}\bigg\},
\end{align*}
where $H^{(k)}(x):=(\mathcal{H}g^{(k)})(x,z), H^{*}(x):=(\mathcal{H}g^{*})(x,z)$.

Letting $\epsilon\rightarrow0$ and $n\rightarrow\infty$, we obtain 
\begin{equation*}
\Vert G^{(k)}(x)-G^{*}(x)\Vert_{W^{2}_{p}([-m,m])}
\le C\Vert H^{(k)}(x)-H^{*}(x)\Vert_{L^{p}([-m,m])}.
\end{equation*}

For a fixed $x$, define $\Gamma^{(k)}_{L}:=\sup\{x'\le x|x'\notin \mathcal{W}^{G^{(k)}}\}\in\mathbb{R}\bigcup\{-\infty\}$, $\Gamma^{(k)}_{R}:=\inf\{x'\ge x|x'\notin \mathcal{W}^{G^{(k)}}\}$, $\Gamma^{*}_{L}:=\sup\{x'\le x|x'\notin \mathcal{W}^{G^{*}}\}\in\mathbb{R}\bigcup\{-\infty\}$, $\Gamma^{*}_{R}:=\inf\{x'\ge x|x'\notin \mathcal{W}^{G^{*}}\}$, then $H^{(k)}$ and $H^{*}$ satisfies the same PDE on $[\Gamma^{(k)}_{L}\vee\Gamma^{*}_{L},\Gamma^{(k)}_{R}\vee\Gamma^{*}_{R}]$ with $\lim\limits_{k\rightarrow\infty}\Gamma^{(k)}_{L}=\Gamma^{*}_{L}$ and $\lim\limits_{k\rightarrow\infty}\Gamma^{(k)}_{R}=\Gamma^{*}_{R}$. For any $\epsilon>0$, for $k$ sufficiently large, the region $\big\{x\big||H^{(k)}(x)-H^{*}(x)|\ge\epsilon\big\}$ is a subset of $\bigcup(\Gamma^{(k)}_{L}\vee\Gamma^{*}_{L},\Gamma^{(k)}_{R}\wedge\Gamma^{*}_{R})$ where the union of open intervals are countable. Applying the maximum principle in each $(\Gamma^{(k)}_{L}\vee\Gamma^{*}_{L},\Gamma^{(k)}_{R}\vee\Gamma^{*}_{R})$, we obtain that $|H^{(k)}(x)-H^{*}(x)|=\epsilon$ on $\big\{x\big||H^{(k)}(x)-H^{*}(x)|\ge\epsilon\big\}$, it then holds that
\begin{equation*}
|H^{(k)}(x)-H^{*}(x)|\le\epsilon,\quad\forall x\in\mathbb{R}.
\end{equation*}
As a result, for any $\epsilon>0$, there exists $k$ sufficiently large such that
\begin{equation*}
\Vert G^{(k)}(x)-G^{*}(x)\Vert_{1+\alpha}\le C\sum\limits_{m=1}^{\infty}\frac{1}{\kappa^{m}}\Vert H^{(k)}(x)-H^{*}(x)\Vert_{L^{p}([-m,m])}\le C\sum\limits_{m=1}^{\infty}\frac{1}{\kappa^{m}}\Vert\epsilon\Vert_{L^{p}([-m,m])}\le C\epsilon,
\end{equation*}
which yields the continuity of $\mathcal{Y}$.

Applying the Schauder's fixed point theorem, we obtain that $\mathcal{Y}$ has a fixed point in $\mathcal{R}_{G}$.

\end{proof}

\begin{corollary}
There exists an equilibrium auxiliary singular control law $\hat{\Upsilon}$.
\end{corollary}
\begin{proof}
For any $z\in[0,1]$, let $G$ be a fixed point of $\mathcal{Y}$ in $\mathcal{R}_{G}$ and $g:=\mathcal{Y}_{g}G$, then $G\in C^{1+\alpha}(\mathbb{R})\bigcap C^{2+\alpha}(\mathbb{R}\setminus\partial\mathcal{W}^{G})$ and $g^{p,s}\in C(\mathbb{R}\times[s,+\infty))\bigcap C^{2+\alpha}((\mathbb{R}\setminus\partial\mathcal{W}^{G})\times[s,+\infty))$ solves (\ref{G})$\sim$(\ref{gp}). Define 
\begin{align*}
&V(x,t,z):=-\frac{\lambda}{\beta}\mathcal{E}(1)-\int_{z}^{1}G(x,z')dz',\\
&f(x,t,z,p,s):=-\frac{\lambda}{\beta}e^{-\beta(t-s)}\mathcal{E}(1)-\int_{z}^{1}g(x,t,z',p,s)dz'.
\end{align*}
Then $V$ and $f$ satisfy the regularity condition (a) and the extended HJB system (b) in Theorem \ref{verif-outer}. The admissibility condition (c) obviously holds. The integrability and transversality conditions (d) and (e) can be easily verified by using the explicit probability representation (\ref{present_g}) and (\ref{present_G_opt}). Invoking Theorem \ref{verif-outer}, we obtain the existence of equilibrium in the outer problem.
\end{proof}

\ \\
\textbf{Acknowledgements}:  Zongxia  Liang is supported by the National Natural Science Foundation of China under grant no. 12271290. Xiang Yu is supported by the Hong Kong RGC General Research Fund (GRF) under grant no. 15211524, the Hong Kong Polytechnic University research grant under no. P0045654 and the Research Centre for Quantitative Finance at the Hong Kong Polytechnic University under grant no. P0042708.

\bibliographystyle{abbrvnat}
{\small
\bibliography{references}}
\end{document}